\documentclass[oneside, a4paper,12pt]{amsart}

\usepackage{geometry}
\usepackage{amsmath}
\usepackage{amsfonts}
\usepackage{amsthm}
\usepackage{amssymb}
\usepackage{graphicx}
\usepackage{array}
\usepackage{enumerate}
\usepackage{hyperref}
\hypersetup{
colorlinks=false,
pdfborder={0 0 0},}

\newcommand{\g}{\mathfrak{g}} % arbitrary Lie algebra
\renewcommand{\sl}{\mathfrak{sl}} % sl Lie algebra
\newcommand{\R}{\mathbb{R}} % reals
\newcommand{\C}{\mathbb{C}} % Complex
\renewcommand{\k}{\mathbb{K}} % finite field

\theoremstyle{plain}
\newtheorem{theorem}{Theorem}[section]
\newtheorem{satz}{Theorem}[section]

\newtheorem{conjecture}[theorem]{Conjecture}
\newtheorem{corollary}[theorem]{Corollary}

\newtheorem{Lemma}[theorem]{Lemma}
\newtheorem{lemma}[theorem]{Lemma}

\newtheorem{definition}[theorem]{Definition}

\newtheorem*{Bsp}{Example}
\newtheorem*{List}{List}

\newtheorem{remark}[theorem]{Remark}

\newcommand{\supp}{\mathrm{supp}}

\newcommand{\pr}{\mathrm{pr}}

\usepackage{keyval}
\makeatletter
\define@key{setpar}{left}[0pt]{\leftmargin=#1}
\define@key{setpar}{right}[0pt]{\rightmargin=#1}
\define@key{setpar}{both}{\leftmargin=#1\relax\rightmargin=#1}
\makeatother

\title[On right coideal subalgebras of quantum groups]
{On right coideal subalgebras of quantum groups}
\author{\hspace{1.5cm}Karolina Vocke, karolina.vocke@gmx.de\newline Phillips-University 
Marburg, \\ Fachbereich Mathematik und Informatik, \\ Hans-Meerwein-Strasse, 35032 Marburg, Germany }

\begin{document}

\begin{abstract}
Right coideal subalgebras are interesting substructures of Hopf algebras such as quantum groups. Examples of right coideal subalgebras are the quantum Borel part as well as quantum symmetric pairs. Classifying right coideal subalgebras is a difficult question with notable results by Schneider, Heckenberger and Kolb. After reviewing these results, as main result we prove that an arbitrary right coideal subalgebras has a particularly nice set of generators. This allows in principle to specify the set of right coideal subalgebras in a given case. 
As application we determine right coideal subalgebras of the quantum groups $U_q(\sl_2)$ and $U_q(\sl_3)$ and discuss their representation theoretic properties.\\

\end{abstract}

\maketitle
\tableofcontents

\newpage

%\author{Karolina Vocke}
%\maketitle

\section{Introduction}

The quantum group $U_q(\g)$ is a deformation of the universal enveloping algebra $U(\g)$ of a semisimple Lie algebra (see e.g. \cite{Jan96}), and it carries the structure of a Hopf algebra. A right coideal subalgebra $C\subset U_q(\g)$ is defined to be a subalgebra $C$ such that $\Delta(C)\subset C\otimes U_q(\g)$. 
While the Hopf subalgebras, fulfilling $\Delta(H)\subset H\otimes H$, are well understood for quantum groups and correspond to what we expect from the Lie algebra, the concept of right coideal subalgebras (RCS) appears to be much more difficult.\\ 

The study of RCS in a broader context has far-reaching applications in Skryabins freeness results and for the structure theory of quantum groups and Nicols algebras. For example in \cite{HS09} Heckenberger and Schneider constructed the PBW basis systematically by a flag of RCS.
 As a very interesting byproduct they prove that homogeneous right coideal subalgebra are in 1:1 correnspondence with the Weyl group of $U_q(\g)$. Moreover the study of RCS yields interesting Lie-theoretic structures like the quantum Borel part $U_q(\g)^+$ and quantum symmetric pairs.\\ 
 
 Recently Heckenberger and Kolb have classified homogeneous RCS (i.e. RCS $C$ with $U^0\subset C$) in \cite{HK11a}, as well as RCS of the positive Borelpart in \cite{HK11b}. Thus broad classes of RCS are already known:
 \begin{center}
\begin{tabular}{|c||c|c|}
\hline
 & $C\subset U_q(\mathfrak{g})^{\geq0}$ & $C\subset U_q(\mathfrak{g})$ \\
\hline\hline
&&\\ 
$U^0\subset C$ & $C=U^+[x]U^0$ & $C=U^+[x]U^0U^+[y]$\\
\hline
&&\\
$U^0\cap C=:T_L$& $(U^+[x]T_L)_\chi$ & general case ?\\
\hline
\end{tabular}
\end{center}~\\

 In Section 2 we review these currently known classification results on RCS.\\

In Section 3 we construct a set of generators for an arbitrary RCS with the property that $C^0:=C\cap U^0$ is a Hopf subalgebra. In particular we show, that one can choose a generating system, consisting of elements whose leading terms lie in $U^{\geq0}$ resp. $U^{\leq0}$. Even more we prove that we can take any generator with at most one $E$- and one $F$-leading term, which are moreover root vectors of $\psi(U^+[w^+]),U^-[w^-]$ for suitably chosen $w^+,w^-\in W$. This statement is our Main Theorem \ref{hauptsatzgenerator}.

% We conjecture in Section 3.4 even more that one can always choose a generating system, which consists of single shifted root vectors as follow: 
$$
\lambda_EE^{\phi_E}_{\mu}+\lambda_F K^{-1}_{\mu-\nu}F^{\phi_F}_{\nu}+\lambda_KK^{-1}_{\mu} 
$$
 \\

In Section 4 we demonstrate our result by determining all RCS of $U_q(\sl_2)$ and $U_q(\sl_3)$ explicitly. We find the
quantum Borel part and its reflections, the quantum symmetric pairs, smaller quantum groups of smaller rank and construct many nontrivial new RCS. With this knowledge we can also classify all so-called Borel subalgebras (i.e. RCS maximal with the property, that any finite dimensional simple representation is onedimensional) of $U_q(\sl_2)$ and $U_q(\sl_3)$.
% 
%!!Hopf subalgebras are all known (sub-Liealgebra), RCS are interesting
%\begin{itemize}
% \item Freeness Results (Skryabin)
% \item Structure theory of quantum groups and Nichols algebras (Schneider, Heckenberger). PBW basis is constructed systematically by a flag of right coideal subalgebras. As a very interesting byproduct they prove that homogeneous right coideal subalgebra are in 1:1 correnspondence with the Weyl group (see section \ref{})
% \item Interesting Lie-Theoretic structures like the quantum Borel part $U_q(\g)^+$ and quantum symmetric pairs
%\end{itemize}
% 
%....
%In section 2 we review the currently known classification results.
%
%.....Kolb/Heckenberger Vierfeldertafel. 
%
%First triangular....Then generator theorem. We note that not all right coideal subalgebras are triangular, but we conjecture \marginpar{prove for $A_n$?} that every right coideal subalgebra is contained in a triangular right coideal subalgebra which is only slightly larger (i.e. of same growth exponent) 
%
%....

\section{Preliminaries}

Let $\g$ be a finite-dimensional semisimple Lie algebra of rank $n$ over the field of complex numbers $\k=\C$.

We denote by $\Pi=\{\alpha_1,\ldots \alpha_n\}$ a set of positive simple roots, by $Q$ the root lattice, and by $\Phi^+\subset Q$ the set of all positive roots. We denote by $(,)$ the symmetric bilinear form on $\R^\Pi$ with the Cartan matrix $c_{ij}=2\frac{(\alpha_i,\alpha_j)}{(\alpha_j,\alpha_j)}$.

Our article is concerned with the quantum group $U=U_q(\g)$ where $q$ is not a root of unity. There exist root vectors $E_\mu$ for all $\mu\in \Phi^+$, constructed by algebra automorphisms $T_w$ due to Lusztig  for each Weyl group element $w\in W$, see \cite{Jan96} Chapter 8.\\ 

A subalgebra $C$ of a Hopf algebra $H$ is called a right coideal subalgebra (RCS) if $\Delta(C)\subset C\otimes H$. Three essential results in the theory of coideal subalgebras of quantum groups are:\\

Call a right coideal subalgebra $C\subset U_q(\g)$ homogeneous iff $U^0\subset C$ (in particular $C$ is homogeneous with respect to the $Q$-grading). 
\begin{theorem}[\cite{HS09} Theorem 7.3]\label{HS09} For every $w\in W$ there is an RCS $U^+[w]U^0$, where $U^+[w]$ is generated by the root vectors $E_{\beta_i}$ for all $\beta_i$ in the subset of roots 
$$\Phi^+(w)=\{\alpha\in \Phi^+\mid w^{-1}\alpha\prec0\}=\{\beta_i\mid i\in\{1,\ldots,\ell(w)\}\}$$
In particular $|\Phi^+(w)|=\ell(w)$, and 
\begin{align}\label{redexpvonteilmenge}
v<w\text{ iff }\Phi^+(v)\subseteq\Phi^+(w)
\end{align}
and for the longest element $\Phi^+(w_0)=\Phi^+$.

Conversely, every homogeneous RCS $C\subset U_q^+(\g)U^0$ is of this form for some $w$.
\end{theorem}
\begin{theorem}[\cite{HK11a} Theorem 3.8]
The homogeneous RCS $C\subset U_q(\g)$ are of the form 
$$C=U^+[w]U^0U^-[v]$$
for a certain subset of pairs $v,w\in W$. 
\end{theorem}

Non-homogeneous RCS are only classified on $U^-U^0$ (or $U^+U^0$):
\begin{theorem}[\cite{HK11b} Theorem 2.15]\label{rcsinu+}
For $w\in W$, let $\phi:U^-[w]\to \k$ be a character and define 
$$\supp(\phi):=\{\beta\in Q\mid \exists x_\beta\in U^-[w] \text{ with }\phi(x_\beta)\neq 0\}$$
Take any subgroup $L\subset \supp(\phi)^{\perp}$, then there exists a character-shifted RCS
$$U^-[w]_\phi:=\{\phi(x^{(1)})x^{(2)}\mid \forall x\in U^-[w]\}$$
and an RCS $U^-[w]_\phi T_L$ with group ring $T_L=\k[L]\subset U^0$.

Conversely, every RCS $C\subset U^-(\g)U^0$ if of this form. 
\end{theorem}

%To construct non-homogeneous RCS $C\subset U_q(\g)$ we shall in the following restrict our attention to: 
%\begin{definition}\label{deftriang}
%We call a right coideal subalgebra \emph{triangular}, if each element splits into elements in $C^{\geq0}$ and $C^{\leq0}$:
 %$$C=(C\cap U^{\geq0})(C\cap U^{\leq0})$$
 %We denote $C^{\geq0}:=C\cap U^{\geq0}$, $C^{\geq0}:=C\cap U^{\leq0}$ and $C^+:=C\cap U^+$ and $C^-:=C\cap U^-$.
%\end{definition}

%FALLS es einen Darstellungsabschnitt geben soll:
% \begin{theorem}[\cite{Jan96} Theorem 5.10]\label{vermaendlich}
% For any dominant integral weight $\lambda\in\Lambda$ the irreducible $U$-module $L(\lambda)$ has finite dimension. 
% Every finite dimensional simple $U$-module is isomorphic to exactly? one $L(\lambda,\sigma)$ with $\lambda\in\Lambda$ dominant.
% %und $\sigma:\mathbb{Z}\Phi\mapsto\{\pm1\}$.
%\end{theorem}
%\begin{theorem}[\cite{Jan96} Theorem 2.9]\label{sl2halbeinfach}
%Let $M$ be a finite dimensional module of $U_q(\sl_2)$ which is the direct sum of its $K$-Eigenspaces. Then $M$ is a semisimple $U$-module, 
%i.e. the direct sum of $L(\lambda,\sigma)$.
%\end{theorem}
% The proof of this Theorem is similar to the Theorem of Weyl for Lie algebras.\\

\paragraph*{Orderings:}\label{ordnungen}
%Totalordnung $<$ auf $\Phi^+(w)$
\index{Ordnung! auf $\Phi^+(w)$}
Let $w\in W$ be fixed with a chosen reduced expression. We consider $\mathbb{N}^{\Phi^+(w)}$ and denote $b\in \mathbb{N}^{\Phi^+(w)}$ by $b=(b_1,\ldots,b_{|\Phi^+(w)|})$ for $b_i\in \mathbb{N}$\\

There is a canonical projection $pr:\mathbb{Z}|\Phi^+(w)|\rightarrow \mathbb{Z}\Pi$ given by the addition on $\Phi^+$, i.e. $pr(b):=b_1\beta_1+\ldots+b_{|\Phi^+(w)|}\beta_{|\Phi^+(w)|}$ and we identify 
$\Phi^+(w)$ with the set of all unit vectors $\mathbb{N}\Phi^+(w)$, i.e. let $\beta_i\in \Phi^+(w)$ be the i-th unit vector in $\mathbb{N}\Phi^+(w)$
\begin{itemize}
\item We define, respecting the coice of a reduced expression of $w$, a total ordering $<$ on $\Phi^+(w)$ by $$\beta_i<\beta_j\in \Phi^+(w) \Leftrightarrow i<j$$
This ordering is convex, i.e. $\mu<\nu\in \Phi^+(w)$ if and only if $\mu<\mu+\nu<\nu$, for $\mu+\nu\in\Phi^+(w)$ see therefore \cite{PA} s.662.

\item In dependence of this total ordering there is another partial ordering on $\mathbb{N}\Phi^+(w)$, 
which is the lexicographical ordering $<_{lex}$ on $\Phi^+(w)$. With this ordering we can order elements
$a,b\in \mathbb{N}^{\Phi^+(w)}$ with $pr(a)=pr(b)$.
\end{itemize}

\begin{Lemma}[\cite{LS91} Prop. 5.5.2]\label{conchi}
Let $w\in W$ be a Weyl group element with reduced expression $w=s_{\alpha_1}\ldots s_{\alpha_{\ell(w)}}$ and \newline $\Phi^+(w)=\{\beta_1,\ldots \beta_{\ell(w)}\}$. 
 For the corresponding root vectors there is the following commutator rule. Let $\beta_i,\beta_j\in \Phi^+(w)$ with $i<j$, then:
 $$[E_{\beta_i},E_{\beta_j}]=\sum_{(a_{i+1},\ldots a_{j-1})\in \mathbb{N}_0^{j-i-1}}m_{(a_{i+1},\ldots a_{j-1})}E_{\beta_{i+1}}^{a_{i+1}}\ldots E_{\beta_{j-1}}^{a_{j-1}} $$
For coefficients $m_{(a_{i+1},\ldots a_{j-1})}\in k$.
 \end{Lemma}

%Hier immerhin Statement hinschreiben als w+,\chi+ w-,\chi- 
%Hast Du dann was interessantes generell in der DISS gesagt? Für An?
%z.B. welche combinationen überhaupt erlaubt sind. Eine Richtung reicht auch!

\section{Constructing a system of generators}\label{main}
\subsection{Summary}

%2017\chapter{Basistheorem}\label{kapitelbasistheorem}
Goal of this part is to introduce a clever choice of an algebra generating system for a right coideal subalgebra $C$ of $U=U_q(\g)$.
%Ziel dieses Kapitels ist die geschickte Wahl eines Algebra-Erzeugendensystems von Rechtscoidealunteralgebren vorzustellen.
That is a set $Z$ of elements in $C$, which generate as an algebra the right coideal subalgebra $C=\langle Z\rangle$ .\\

We restrict to the case of a right coideal subalgebra $C$, for which
$$C^0:=C\cap U^0$$ 

is a sub Hopf algebra.
In particular we show, that one can choose a generating system, consisting of elements whose leading terms lie in $U^{\geq0}$ resp. $U^{\leq0}$.
We conjecture that one can choose a generating system which consists of elements of the following form: 
\begin{align}
\lambda_EE^{\phi_E}_{\mu}+\lambda_F K^{-1}_{\mu-\nu}F^{\phi_F}_{\nu}+\lambda_KK^{-1}_{\mu} \label{ziel}
\end{align}
for root vectors $E_{\mu}$ resp. $F_{\nu}$ in $\psi(U^+[w_1])$ resp. $U^-[w_2]$ respecting a specieal reduced expression of the elements $w_1$ and $w_2\in W$ 
with characters $\phi_E,\phi_F$ on the subalgebras $\psi(U^+[w_1])T_L$ resp. $U^-[w_2]T_L$ for some suitable $L\subset Q$, where $E^{\phi_E}_{\mu}:=(\phi_E\otimes id)\Delta(\psi(E_{\mu}))$ 
and $F^{\phi_F}_{\nu}:=(\phi_F\otimes id)\Delta(F_{\mu})$.\\

%Let $C$ be a right coideal subalgebra with the property $C^0:=C\cap U^0$ is a Hopf subalgebra.
We denote the $U^0$-\emph{part} of $C$ by $T_L:=C\cap U^0$ for $L\subset Q$ a subgroup. 
First we choose a set of $\mathrm{ad}(T_L)$-weight vectors, generating $C$ as algebra, and show, 
that we can choose elements in $U^+_{\mu}U^0_{\nu}U^-_{\gamma}$ with $\nu-\gamma$ constant. 

Then we prove, that we can choose elements with exactly one leading term in $U^{\geq0}$ or $U^{\leq0}$ respectively.
It is even possible to choose root vectors as leading terms respecting a Weyl group element $w\in W$ with a special reduced expression. 

Finally, we provide a proof of the assertion that a choice of generating elements in $(U^{\geq0}+U^{\leq0})\cap C$ is possible 
and prove some further properties,
which hold for $\phi^E,\phi^F$ in (\ref{ziel}).\\

First we prove statements, which hold for any choice of a reduced expression of the maximal element $w_0\in W$.
Fix therefore in the following a reduced expression of the maximal element $w_0$ and consider elements of $U$ in the PBW-basis consisting of root vectors in $U^+[w_0]U^0U^-[w_0]$.
 
For elements in $\mu\in Q$ we use the partial ordering $\prec$ on $Q$ and on elements in $\overline{\mu}\in\mathbb{N}^{\Phi^+}$ the ordering 
$<_{lex}$. We also use the projection
$\pr:\mathbb{Z}\Phi^+\rightarrow \mathbb{Z}\Pi$ 
given by addition in $Q=\mathbb{Z}\Pi$, i.e. $\pr(\overline{\mu}):=\mu_1\beta_1+\ldots+\mu_{|\Phi^+|}\beta_{|\Phi^+|}$.
For every $\overline{\mu}$ there is a unique PBW-Monomial $E_{\overline{\mu}}\in U^+$ (the same holds in $U^-$) given by:
$$E_{\overline{\mu}}:=E_{\beta_{|\Phi^+|}}^{\mu_{|\Phi^+|}}\cdots E_{\beta_1}^{\mu_1}$$

The same holds in $U^-$. We need some further technical notations:
\begin{definition}
 For $\mu\in Q_+,\nu \in Q,\gamma\in Q_+$ we denote the components of an element $X\in U$ 
 in $Q^3$-degree in the triangular decomposition $U^+_\mu U^0_\nu U^-_{-\gamma}$ by $\pr_{(\mu,\nu,\gamma)}(X)$, i.e.
 $$\pr_{(\mu,\nu,\gamma)}:\; U\to U^+_\mu K_\nu U^-_{-\gamma}$$
 $$X_{\mu,\nu,\gamma}:=\pr_{(\mu,\nu,\gamma)}(X)=
\sum_{\substack{\overline{\mu}\in\mathbb{N}^{\Phi^+}\text{ with }pr(\overline{\mu})=\mu\\ \overline{\gamma}\in\mathbb{N}^{\Phi^+}\text{ with }pr(\overline{\gamma})=\gamma}}\limits 
c_{\overline{\mu},\nu,\overline{\gamma}}E_{\overline{\mu}}K_{\nu}F_{\overline{\gamma}}
 $$
 This projection in the first two components $pr_{(\_,\_,0)}$ corresponds to $\pr_{\mu,\nu}$ which is defined in $U^{\geq0}$ as unique $Q^2$-graded projection $pr_{(\alpha,\beta)}:U^{\geq0}\rightarrow U_{\alpha}^+K_{\beta}$. 
\end{definition}
$U$ is not a $Q^3$-graded Hopfalgebra, nevertheless we know:
\begin{align}
\begin{split}
\Delta(X_{\mu,\nu,\gamma})\subset & X_{\mu,\nu,\gamma}\otimes K_{\nu-\gamma}+K_{\mu+\nu}\otimes X_{\mu,\nu,\gamma}+\\&\sum_{\substack{(\mu',\gamma')\\ \text{ with }(0,0)\prec(\mu',\gamma')\prec (\mu,\gamma)}}U^+_{\mu'}K_{\nu+\mu-\mu'}U^-_{-\gamma'}\otimes U^+_{\mu-\mu'}U^0_{\nu-\gamma'}U^-_{-\gamma+\gamma'}
\end{split}
\label{coprodukt}
\end{align}
\begin{definition}
The \emph{$U^{\geq0}$-part} resp. the \emph{$U^{\leq0}$-part} resp. the \emph{mixed part} of an element $X$ are given by the following sums
$$
\sum_{\mu\in Q_+,\nu\in Q}\limits X_{\mu,\nu,0},\qquad 
\sum_{\nu\in Q,\gamma\in Q_+}\limits X_{0,\nu,\gamma},\qquad
\sum_{\mu\neq0,\nu\in Q,\gamma\neq0\in Q_+}\limits X_{\mu,\nu,\gamma}
$$
\end{definition}
\begin{definition}
 For an element $X_{\mu,\nu,\gamma}=
\sum_{\substack{\overline{\mu}\in\mathbb{N}^{\Phi^+}\text{ with }pr(\overline{\mu})=\mu\\ \overline{\gamma}\in\mathbb{N}^{\Phi^+}\text{ with }pr(\overline{\gamma})=\gamma}}\limits 
c_{\overline{\mu},\nu,\overline{\gamma}}E_{\overline{\mu}}K_{\nu}F_{\overline{\gamma}}
 \in U$ and given $\overline{\mu},\overline{\gamma}\in \mathbb{N}^{\Phi^+},\;\nu\in Q$ we denote the coefficient of the PBW-basis elements $E_{\overline{\mu}}K_\nu F_{\overline{\gamma}}$ in the PBW-Basis by
 $$c_{\overline{\mu},\nu,\overline{\gamma}}(X)\in k$$
\end{definition}

%\subsection{Choice of a generating system of $C$}

\subsection{Generating systems with $T_L$-stable root vectos}

\begin{lemma}
\label{konstant}
Let $C$ be an RCS in $U$. We can choose a generating system $Z$ of $C$, such that for all elements 
%$(id\otimes \epsilon)\Delta(X-X\otimes K_{\mu})=0$ für ein $\mu\in\Phi$ und $\epsilon(U^0)=1$ sonst 0 insbesondere gilt für solche 
$X=\sum_{\nu\in Q,\mu,\gamma\in Q_+}\limits X_{\mu,\nu,\gamma}\in Z$ the difference $\nu-\gamma$ is constant.
\end{lemma}

\begin{proof}
We use the fact, that $C$ is an RCS, the formula (\ref{coprodukt}) and the PBW-basis: 
Let for any $\eta\in Q$ the Linear form $\varphi\in U^*$ be given by $\varphi(K_\eta)=1$ and $\varphi(x)=0$ for any other PBW-Basis element. Then
for any element $X\in C$ the following elements $X^{(\eta)}$ lie again in the right coideal subalgebra $C$: 
$$X^{(\eta)}:=(id\otimes \varphi)\Delta(X)
=\sum_{\substack{\mu,\gamma\in Q_+,\nu\in Q\\ \text{with } \nu-\gamma=\eta}} X_{\mu,\nu,\gamma}
$$
On the other hand $X=\sum_{\eta} X^{(\eta)}$, thus we can refine a given generating system $Z$, by replacing all generators 
$X$ by the set of $X^{(\eta)}$.
\end{proof}
%Wir definieren für Elemente $X=\sum_{\mu,\nu,\gamma}X_{(\mu,\nu,\gamma)}$ neue Elemente: $X_c:=\sum_{\mu,\nu,\gamma}X_{(\mu,\nu,\gamma)}K_{c-\mu-\nu-\gamma}$ ist $K_c$-costabil.
\begin{definition}
A \emph{weight vector} \index{Gewichtsvektor} resp. $T_L= U^0\cap C$ is an element $x\in C$, such that there is a
$\mu\in Q$, with $\mathrm{ad}(K_{\nu})x=q^{(\mu,\nu)}x \forall\nu\in L$.

\end{definition}
The adjoint action of elements in $T_L$ on $C$ is defined, because $T_L$ is a sub Hopfalgebra.
In particular we can choose a generating system consisting of weight vectors of $T_L$, as $C$ is $\mathrm{ad}T_L$-stable and as $adU^0$ is diagonalizable on $U$.
 
%\section{Leitterme von erzeugenden Elementen von $C$}

\subsection{Leading terms of elements in $C$}
\begin{definition}\label{def_Leitterm}
Given an element $X$, we define the set of all \emph{leading terms}\index{Leitterm} as follows: \newline
First we define the set of all maximal $U^+$-degrees:
$$\mathcal{E}(X):=\mathrm{max}_{\prec}\{ 0\neq\mu\in Q_+ \mid\exists \nu\in Q,\gamma\in Q_+\text{ with }X_{\mu,\nu,\gamma}\neq0\}$$

% Beispiel ist $\mathcal{E}(F_i)=0$ für alle $F_i\in U^{\leq0}$. 
For any $\mu\in \mathcal{E}(X)$ the corresponding \emph{E-leading term} is
$$L_{\mu}(X):=\sum_{\nu\in Q,\gamma\in Q_+}\limits X_{\mu,\nu,\gamma}$$ 
Similarly we can define \emph{F-leading terms}.
\end{definition}
\begin{remark}
If an element $X\in C$ has no E-leading terms, i.e. $\mathcal{E}(X)=\emptyset$, then $X\in U^{\leq0}$.
\end{remark}
\begin{theorem}
\label{Leitterme}
Let $C$ be an RCS with the property $C^0:=C\cap U^0$ is a sub Hopfalgebra. 
We can choose a generating system of $C$ consisting of elements whose E-leading terms lie in $U^{\geq0}$.
\end{theorem}

\begin{proof}
For elements $X$ in $C$ let the set of PBW-Monomials of $E$-leading terms with non trivial $F$-part be defined as:
$$\mathcal{M}(X):=\left\{\overline{\mu}\in \mathbb{N}^{\Phi^+ }
\mid\exists \nu\in Q,0\neq \overline{\gamma}\in\mathbb{N}^{\Phi^+}
\text{ such that }c_{\overline{\mu},\nu,\overline{\gamma}}(X)\neq0
\text{ und }\pr(\overline{\mu})\in\mathcal{E}(X)\right\}$$
%=\mu\in Q\text{ sodass } \mu \text{ maximal in }\{\mu\in Q\mid\exists \nu\in Q,\gamma\in Q_+\text{ mit }X_{\mu,\nu,\gamma}\neq0\}\text{ bzgl der Ordnung }\prec\\
%\text{ und }L_{\mu}\notin U^{\geq0}}\}$$
%die Menge der Wurzeln $\mu$ für die es Leitterme $L_{\mu}$ gibt mit $L_{\mu}\notin U^{\geq0}$. 
Similarly we define for a set $Z$ of elements $X\in C$ the set
$\mathcal{M}(Z):=\bigcup_{X\in Z} \mathcal{M}(X)$.
For comparing elements and generating systems respecting their corresponding set $\mathcal{M}$, we define a partial ordering $\leq$ on subsets of $\mathbb{N}^{\Phi^+ }$ as follows:

$\mathcal{N}\leq \mathcal{M}$ if and only if there exists for all $\overline{\nu}\in \mathcal{N}$ a $\overline{\mu}\in \mathcal{M}$ , such that either
 $\pr(\overline{\nu})\prec \pr(\overline{\mu})$ (i.e. the partial ordering on $Q$) or 
if $\pr(\overline{\nu})=\pr(\overline{\mu})$ then $\overline{\nu}\leq_{lex} \overline{\mu}$ (i.e. the lexicographical ordering of PBW-Monomials). We call $\mathcal{N}<\mathcal{M}$ strictly smaller, if there exists at least one $\overline{\mu}\in \mathcal{M}$ which is not less than or equal (again respecting partial and then lexicographical ordering) to {\it all} elements $\overline{\nu}\in \mathcal{N}$. \\

The goal of the Theoremis to prove, that there exists a generating system $Z$ of $C$ with $\mathcal{M}(Z)=\emptyset$.
In the following we give an algorithm, to construct for an element $X\in C$ with $\mathcal{M}(X)\neq \emptyset$ 
finitely many new elements $\mathcal{Y}=\{Y_1,\ldots Y_n\}\subset C$ with 
$\mathcal{M}(\{Y_1,\ldots Y_n\})<\mathcal{M}(X)$ and $X\in \langle \mathcal{Y}\rangle$, i.e. $X$ lies in the subalgebra generated by the elements $Y_i$. 
The sets $\mathcal{M}(X)$ for single elements $X$ or finite sets of elements are finite. 

By the definition of $\mathcal{N}<\mathcal{M}$, starting with a set $\mathcal{M}(X)$, in every iteration step either one of the maximal $Q$-degrees of the set decreases or (for similar $Q$-degree) in one $Q$-degree the lexicographical maximal degree decreases. 

This proves, that, starting with a $\mathcal{M}(X)$ we can construct in finitely many steps a finite set $\mathcal{Y}\subset C$ , for which $\mathcal{M}(\mathcal{Y})=\emptyset$ and $X\in\langle\mathcal{Y}\rangle$. Finally we conclude, that even for an a-priori infinite generating system $Z$ of $C$ we can construct the desired alternative generating system $Z'$ with $\mathcal{M}(Z')=\emptyset$ as union of all $\mathcal{Y}$ for all $X\in Z$; thus the claim follows.\\

Now we give the announced algorithm, to construct for an element $X\in C$ with $\mathcal{M}(X)\neq \emptyset$ 
finitely many elements $\mathcal{Y}=\{Y_1,\ldots Y_n\}\subset C$ with $\mathcal{M}(\{Y_1,\ldots Y_n\})<\mathcal{M}(X)$
and $X\in \langle \mathcal{Y}\rangle$:\\

Let $\overline{\mu_{max}}\in \mathcal{M}(X)$ be a maximal element (maximal respecting the partial ordering $\prec$ and for all elements with the same $Q$-degree 
respecting the lexicographical ordering). We consider the term
$X_{\overline{\mu_{max}}}=
E_{\overline{\mu_{max}}}
\sum_{\nu\in Q,\overline{\gamma}\in \mathbb{N}^{\Phi^+}}
\limits c_{\overline{\mu_{max}},\nu,\overline{\gamma}}
K_{\nu}F_{\overline{\gamma}}$. 

Due to Lemma \ref{konstant} we can choose $X$ in a way, that $\nu-\gamma=\eta$ holds for a constant $\eta\in Q$. 
Thus
$X_{\overline{\mu_{max}}}=E_{\overline{\mu_{max}}}\sum_{\overline{\gamma}\in \mathbb{N}^{\Phi^+}}\limits c_{\overline{\mu_{max}}, \eta+\pr(\overline{\gamma}),\overline{\gamma}}K_{\eta+\pr(\overline{\gamma})}F_{\overline{\gamma}}$.\\

In the following let $\gamma=\pr(\overline{\gamma})$, $\overline{\mu}:=\overline{\mu_{max}}$ and $\mu=\pr(\overline{\mu})$. We can use the comultiplication to prove with the right coideal property of $C$ for all $\overline{\gamma}\in \mathbb{N}^{\Phi^+}$
\begin{equation}\sum_{\overline{\gamma}}c_{\overline{\mu}, \eta+\gamma,\overline{\gamma}}K_{\eta+\gamma+\mu}F_{\overline{\gamma}}\in C\end{equation}

It is for all $\gamma$ in $Q_+$:
$$\mathcal{M}(K_{\eta+\gamma+\mu})=\emptyset=\mathcal{M}(K_{\eta+\gamma+\mu}F_{\overline{\gamma}})$$

Let $\gamma_{max}=\gamma_{max}(X,\overline{\mu_{max}})$ be maximal in $ Q_+$ respecting $\prec$ such that there exists a $\overline{\gamma_{max}}$   with $\pr(\overline{\gamma_{max}})=\gamma_{max}$ such that $c_{\overline{\mu_{max}}, \eta+\gamma,\overline{\gamma}}\neq0$. There exists such a $\gamma_{max}$, because $\overline{\mu_{max}}\in \mathcal{M}(X)$.
Consider now the coproduct of $X$: 

$$\Delta(X)\subset \sum_{\gamma\in Q_+}\limits K_{\eta+\gamma_{max}+\mu}\otimes L_{\mu}+U\otimes U_{x\prec\mu}^+U^0U^-+\text{ terms in }U^0(U^-\cap ker\epsilon)\otimes U_{\mu}^+U^0U^-$$ 
thus also 
%für das maximale $\gamma\in Q_+$, sodass $c_{\overline{\mu}_{max},\nu,\overline{\gamma}}\neq0$ für
%ein $\overline{\gamma}$ mit $\pr(\overline{\gamma})=\gamma$
% $X_{\mu,\eta+\gamma,\gamma}\neq0$
$K_{\eta+\gamma_{max}+\mu}$ lies in $C$ and, as $C^0$ is a sub Hopfalgebra, in particular also
\begin{align}
K_{\eta+\gamma_{max}+\mu}^{-1}\in C\end{align}

On the other hand we get from the right coideal property for this $\gamma_{max}$
\begin{align}
E_{\overline{\mu}}K_{\eta+\gamma_{max}}+(rest)\in C\end{align}
with $(rest):=\sum_{\mu'\prec\mu\in Q_+,\gamma\in Q_+}\limits(rest)_{\mu', \eta+\gamma,\gamma}+\sum_{\overline{\gamma},\overline{\mu'}<\overline{\mu},\mu'=\mu}c_{\overline{\mu}, \eta+\gamma,\overline{\gamma}}^{(rest)}E_{\overline{\mu'}}K_{ \eta+\gamma}F_{\overline{\gamma}}$. 
%$(rest)$ in $U^+_{\prec\mu}U^0U^-+\sum_{\mu'\in M(X)}U^+_{\preceq\mu'}U^0U^-$ liegt, 
Thus
$$\mathcal{M}(E_{\overline{\mu}}K_{\eta+\gamma_{max}}+(rest))<\mathcal{M}(X)$$ 
as in particular $\overline{\mu}_{max}\notin \mathcal{M}(E_{\overline{\mu}}K_{\eta+\gamma_{max}}+(rest))$. 

Summing over the different $\overline{\gamma}$ and considering the product of these elements yields:
\begin{align*}
z:=&\underbrace{(E_{\overline{\mu}}K_{\eta+\gamma_{max}}+(rest)}_{(6)}\underbrace{K_{\eta+\gamma_{max}+\mu}^{-1}}_{(5)}\underbrace{\sum_{\overline{\gamma}}\limits c_{\overline{\mu}, \eta+\gamma,\overline{\gamma}}K_{\eta+\gamma+\mu}F_{\overline{\gamma}}}_{(4)}\\
&=X_{\overline{\mu}}+\sum_{\overline{\gamma}}\limits (rest) c_{\overline{\mu}, \eta+\gamma,\overline{\gamma}}K_{-\gamma_{max}+\gamma}^{-1}F_{\overline{\gamma}}\in C
\end{align*} 

Thus for $z$ holds $\mathcal{M}(z)\leq\mathcal{M}(X)$  and $\mathcal{M}(X-z)\leq\mathcal{M}(X)$ .\\ 

Obviously we can generate $X$ by the elements $z$ and $X-z$, where $z$ can be generated by elements $E_{\overline{\mu}}K_{\eta+\gamma_{max}}+(rest)$, $K_{\eta+\gamma_{max}+\mu}$ and $K_{\eta+\gamma+\mu}F_{\overline{\gamma}}$ with $\mathcal{M}(E_{\overline{\mu}}K_{\eta+\gamma_{max}}+(rest)),\mathcal{M}(K_{\eta+\gamma_{max}+\mu}),
\mathcal{M}(K_{\eta+\gamma+\mu}F_{\overline{\gamma}})<\mathcal{M}(X)$. \\

Let $G$ be the finite set of all $\overline{\gamma}$ with $c_{\overline{\mu}_{max},\eta+\gamma,\overline{\gamma}}\neq 0$. 
Assume $\mathcal{M}(X-z)<\mathcal{M}(X)$ (in particular if $\overline{\mu}\notin \mathcal{M}(X-z)$): 
Then we have constructed the desired set in $C$ as follows:
\begin{align*}
\mathcal{Y}
=&\bigcup_{\overline{\gamma}\in G}\left\{\vphantom{x^{x^{x^x}}}
E_{\overline{\mu}}K_{\eta+\gamma_{max}}+(rest),\;
K_{\eta+\gamma_{max}+\mu},\;
K_{\eta+\gamma+\mu}F_{\overline{\gamma}}
\right\}\\
&\;\cup \;
\left\{
X-\underbrace{\sum_{\overline{\gamma}\in G}
\left(E_{\overline{\mu}}K_{\eta+\gamma_{max}}+(rest)\right)
\cdot K_{\eta+\gamma_{max}+\mu}^{-1}
\cdot K_{\eta+\gamma+\mu}F_{\overline{\gamma}}}_{=z}
\right\}
\end{align*}
such that $\langle\mathcal{Y}\rangle\ni X$ and $\mathcal{M}(\mathcal{Y})<\mathcal{M}(X)$. As stated above, we can thus inductively replace the generating system $Z$ by a generating system $Z'$ with $\mathcal{M}(Z')=\emptyset$, and thus, as claimed, construct a generating system consisting of elements with
$E$-leading terms in $U^{\geq 0}$, so the claim is proven.\\

Assume $\mathcal{M}(X-z)=\mathcal{M}(X)$, then in particular $\overline{\mu}\in \mathcal{M}(X-z)$ 
(maximal respecting the partial ordering $\prec$ and among all elements with the same $Q$-degree 
maximal respecting the lexicographical ordering). We construct a series of elements: $X^{(i)}$ with $X^{(1)}:=X-z$ and $X^{(i)}:=X^{(i-1)}-z^{(i-1)}$ for $i>1$, by considering for any $X^{(i)}$ with $\mathcal{M}(X-z)=\mathcal{M}(X)$ (i.e. that is
$\overline{\mu}\in \mathcal{M}(X^{(i)})$) the decomposition above. Thus there exists a $j$, such that $\mathcal{M}(X^{(j)}-z^{(j)})<\mathcal{M}(X)$ and $X$ can be generated by elements
 $X^{(j)}-z^{(j)}$ and $z^{(i)}$ for all $i\leq j$, where for all $i\leq j$ holds $\mathcal{M}(z^{(i)})<\mathcal{M}(X^{(i)})\leq \mathcal{M}(X)$.
Thus, as above, we can inductively construct a generating system of the required form.\\ 

\end{proof}

\begin{lemma}
Let $C$ be an RCS as above, $X\in C$ a weight vector of $T_L$, then $X$ has at most one leading term in $U^{\geq0}$.
\end{lemma}

\begin{proof}
Assume a weight vector $X$ of $T_L$ has two leading terms $L_{\mu}$ and $L_{\nu}$ for different elements
$\mu,\nu\in Q_+$, then from Lemma \ref{konstant} follows similar to the proof of Theorem \ref{Leitterme}, that due to 
$$L_{\mu}=\sum_{\overline{\mu}\text{ with }\pr(\overline{\mu})=\mu}
\limits c_{\overline{\mu},\eta,0}E_{\overline{\mu}}K_\eta$$ 
$$L_{\nu}=\sum_{\overline{\nu}\text{ with }\pr(\overline{\nu})=\nu}
\limits c_{\overline{\nu},\eta,0}E_{\overline{\nu}}K_\eta$$ 
also $K_{_\eta+\mu}\in C$ and $K_{_\eta+\nu}\in C$, 
so in particular as $C^0:=C\cap U^0$ is a sub Hopfalgebra, also $K_{\mu-\nu}$ lies in $C$. 
From the stability under $T_L$ it follows that $\mathrm{ad}(K_{\mu-\nu})X=q^{(\mu-\nu)}X$. 
This is a contradiction to $(\mu-\nu,\mu)\neq(\mu-\nu,\nu)$.
\end{proof}

With the same argument we can normalize the leading term, such that $L_{\mu}=X_{\mu}K_{\mu}^{-1}$ for a $X_{\mu}\in U_{\mu}^+$ and $K_{\mu}^{-1}\in U^0$.
The theorems above also hold for $F$-leading terms and we can choose a generating system of elements $X$ which have a maximum of one $E$-leading term $L_\mu\in U^{\geq0}$ and one F-leading term $L^-_{\nu} \in U^{\leq0}$,
because the two construction steps can not (or only conditionally) reverse each other since they both reduce the overall degree of the mixed terms.

\begin{corollary}
Let $C$ be an RCS with the property $C^0:=C\cap U^0$ is a sub Hopfalgebra. 
We can choose a generating system of $C$ consisting of elements with each at most one $E$- and one $F$-leading term $L_\mu,L_{-\nu}$, which are moreover of the form:
$$L_\mu=X_\mu K_\mu^{-1},\; X_\mu\in U^+_\mu \qquad L_{-\nu}=Y_{-\nu}K_{\mu-\nu}^{-1},\;Y_{-\nu}\in U^-_{-\nu}$$
\end{corollary}

For leading terms $L_{\mu}$ in $U^{\geq0}$ holds, that the product of two such leading terms is again the leading term of an element in $ C $.
Moreover, for any Linear form $\varphi\in U^*$ also $(id\otimes \varphi)\Delta(L_{\mu})$
is a leadingterm of an element in $C$. The set of all leading terms forms again an RCS $\mathfrak{L}$ in $U^{\geq0}$. 
From Theorem \ref{rcsinu+} there exists a $w\in W$ such that $\mathfrak{L}=U^+[w]T_{L'}$. 
Thus we can choose exactly those elements in the generating system whose leading terms are root vectos and so in particular generate all elements in $C$ except $C^{\leq0}$. 
 From these considerations, the main result of this paper follows:

%\begin{theorem}
% Somit gilt für die Rechtscoidealunteralgebra die aus den Leittermen in $U^+$ des EZS von $C$ erzeugt wird: $U^+[w]$, dass alle PBWBasiselemente wieder Leitterme von $X\in C$ sind. 
%\end{theorem}

%\begin{proof}
%theoretisch einfach, technisch schwer
%\end{proof}

\begin{theorem}
\label{pbwleitterm}\label{hauptsatzgenerator}
Let $C$ be an RCS with the property $C^0:=C\cap U^0$ is a sub Hopf--algebra. 
Then we can construct a generating system of $C$, such that any generator has at most one $E$- and one $F$-leading term, 
which is moreover a root vector of $\psi(U^+[w^+]),U^-[w^-]$ with a suitable $U^0$-part for suitably chosen $w^+,w^-\in W$.
\end{theorem}

\begin{corollary}\label{leittermwahl}
$C^{\geq0}:=U^{\geq0}\cap C$ is an RCS of $C$ so by Theorem \ref{rcsinu+} of the form 
$\psi(U^+[w'])_{\phi}T_L$ for $w'\in W$ a Weylgroup element and a character $\phi$ on $\psi(U^+[w'])$. 
Then $w'\leq w$ for $w$ chosen above.
\end{corollary}

\subsection{Further conjecture of a generating system in $U^{\geq0}+U^{\leq0}$}

We conjecture that we can substantially intensify the previous theorem, which we neither prove here nor use in the following:
\begin{conjecture}\label{conchierweiterung}
Let $w\in W$ be a Weyl group element with reduced expression $w=s_{\alpha_1}\ldots s_{\alpha_{\ell(w)}}$ and $\Phi^+(w)=\{\beta_1,\ldots \beta_{\ell(w)}\}$. 
 For the corresponding root vectors we know by \ref{conchi} that the following commutator rule holds: Let $\beta_i,\beta_j\in \Phi^+(w)$, and $i<j$, then:
 $$[E_{\beta_i},E_{\beta_j}]=\sum_{(a_{i+1},\ldots a_{j-1})\in \mathbb{N}_0^{j-i-1}}m_{(a_{i+1},\ldots a_{j-1})}E_{\beta_{i+1}}^{a_{i+1}}\ldots E_{\beta_{j-1}}^{a_{j-1}} $$
 by \cite{LS91} Prop. 5.5.2.\\
 
 We conjecture additionally, that for $\beta_i,\beta_j\in\Phi^+$ with $\beta_i+\beta_j=\beta_k\in \Phi^+$ for the vector $(a_{i+1},\ldots a_{j-1})$ with $a_i=\delta_{ik}$ 
 the coefficient $m_{(a_{i+1},\ldots a_{j-1})}$ is not zero.
\end{conjecture}
\begin{conjecture}\label{basistheorem}
We can choose a generating system of $C$ consisting of elements of the form 
$$\lambda_EE^{\phi_E}_{\mu}+\lambda_F K^{-1}_{\mu-\nu}F_{\nu}^{\phi_F}+\lambda_KK^{-1}_{\mu}$$
Where $E_{\mu}$ resp. $F_{\nu}$ are root vectors respecting $\psi(U^+[w])$, resp. $U^-[w']$ for certain Weyl group elements $w,w'\in W$
and characters $\phi_E:\psi(U^+[w])\rightarrow \mathbb{C}$, $\phi_F:U^-[w']\rightarrow \mathbb{C}$.
\end{conjecture}

% \begin{conjecture}\label{u+u-}
% Wir können ein Erzeugendensystem von $C$ wählen aus Elementen $X$ deren gemischter Anteil leer ist. (d.h. $X\in (U^{\geq0}+U^{\leq0})\cap C$).
% \end{conjecture}

\begin{remark}
A proof would essentially work as the proof of Theorem \ref{Leitterme}. 
But at a crucial point the special property of the ordering $\Phi^+(w)$ from conjecture \ref{conchierweiterung} is used:
Assume $X$ has not the desired form. We consider first a maximal mixed summand resp. $\prec$ and argue as in the proof above,
however instead of the global maximality of the $E-Terms$ we use conjecture
\ref{conchierweiterung} and argue, that this element can be replaced by other generating elements with a smaller leading term. 
\end{remark}%\section{E-Teil als Charakterverschiebung von Leitterm}
For all right coideal subalgebras with a generating system of the form above, we can prove some further restrictions of the $U^0$-part and the support $supp(\phi_E)$ resp. $supp(\phi_F)$.

\begin{lemma}\label{muundnu}
For generating elements of an RCS $C$ of the form 
$$\lambda_EE^{\phi_E}_{\mu}+\lambda_F K^{-1}_{\mu-\nu}F_{\nu}^{\phi_F}+\lambda_KK^{-1}_{\mu}$$ 
with $\mu\neq \nu$, the following restriction holds:
\begin{itemize}
\item $\lambda_K=0$ or $\lambda_F=0$
\item $\mu+\nu\bot \mu-\nu$
\end{itemize} 
\end{lemma}

\begin{proof}
Assume $\lambda_F\neq0$. As $L_{\mu}$ and $L^-_{\nu}K^{-1}_{\mu-\nu}$ are weight vectors it follows, that $K_{\mu-\nu}\in C$. 
The claims follow easy from the fact, that the element is an $\mathrm{ad}(T_L)$-weight vector, i.e. in particular $\mathrm{ad}K_{\mu-\nu}$. 
Thus from $\lambda_K\neq0$ follows $(\mu-\nu,\mu)=0$, this is a contradiction to $\mu\neq0$ and $\lambda_F\neq0$. The second claim, follows in the same way, as $(\mu-\nu,\mu)=-(\mu-\nu,\nu)$ due to the ad$T_L$-stability.
\end{proof}

\begin{lemma}\label{0part}
For the $U^0$-part $T_L$ of an RCS $C$ 
which is generated by elements of the form $\lambda_EE_{\mu}^{\phi_E}K_{\mu}^{-1}+\lambda_F K^{-1}_{\mu-\nu}F_{\nu}^{\phi_F}+\lambda_KK_{\mu}^{-1}$ with $\phi_E$, and $\phi_F$ as above $L\subset (supp(\phi_E)\cup supp(\phi_F))^{\bot}$
and $L\bot K_{\mu+\nu}$. 
%Sowie für alle $\mu\in supp(\phi_E)\cup supp(\phi_F)$ für die es ein Elememt $L^+_{\mu}+\lambda_KK_{\mu}+\lambda_FL^-_{\nu}K_{\mu-\nu}$ mit $\lambda_K\neq 0$ in $C$ gibt, gilt auch $L\bot \mu$.
\end{lemma}

\begin{proof}
The claims follow from the ad$T_L$ stable choice of the generating system similar to Lemma \ref{muundnu}.

%??Die zweite Behauptung gilt, weil die erzeugenden Elemente der obigen Form als $T_L$-Gewichtsvektoren gewählt wurden.

% die erzeugenden Elemente als $ad(T_L)$-Gewichtsvektoren gewählt waren und für alle Elemente $X\in C^{\geq0}$ mit Leitterm in $\mu\in supp(\phi'_E)$ ein $X$ von der Form $L_{\mu}+\lambda_KK_{\mu}+y$ existiert und $\lambda_K\neq 0$ also für $K_{\nu}\in C^0$ mit $\nu$ nicht $\bot\mu$ folgt ein Widerspruch. Für allgemeine $X\in C$...
%Zeige $T_L\in \bot supp(\phi)$

\end{proof}

\section{Examples}
%Jeweils Kommentar zur Darstellungstheorie
\subsection{Right coideal subalgebras of $U_q(\sl_2)$}
%\chapter{Borelunteralgebren der $U_q(\sl_2)$}\label{uqsl2}
\subsubsection*{Data of $U_q(\sl_2)$}
The smallest example of a qantum group is $U_q(\sl_2)$. It is the universal enveloping algebra of the Lie algebra $\sl_2$ with root system $\Phi=\{\alpha,-\alpha\}$. It is generated as an algebra by the elements $E_{\alpha},F_{\alpha},K_{\alpha},K_{\alpha}^{-1}$ with the following relations 
$$[E_{\alpha},F_{\alpha}]_1=
\frac{K_{\alpha}-K_{\alpha}^{-1}}{q-q^{-1}}\quad [E_{\alpha},K_{\alpha}]_{q^{-2}}=[F_{\alpha},K_{\alpha}]_{q^2}=0$$
where $[x,y]_q:=xy-qyx$ is the q-commutator. $U_q(\sl_2)$ is equipped with the following Hopfalgebra structure
 \begin{align*}
 \Delta(E_{\alpha})=E_{\alpha}\otimes 1+K_{\alpha}\otimes E_{\alpha}\quad 
  \Delta(F_{\alpha})=F_{\alpha}\otimes K_{\alpha}^{-1}+1\otimes F_{\alpha}\quad
  \Delta(K_{\alpha})=K_{\alpha}\otimes K_{\alpha}
  \end{align*}
For simplicity we write in the following $E,F,K$ and $K^{-1}$ for the generators $E_{\alpha},F_{\alpha}$, $K_{\alpha},K_{\alpha}^{-1}$.

Due to chapter \ref{main} a list of possible generating elements is up to symmetry in E and F given by:
\begin{enumerate}
\item $K^i$
\item $EK^{-1}$
\item $EK^{-1}+\lambda K^{-1}$
\item $EK^{-1}+c_FF +c_KK^{-1}$
\end{enumerate}

\subsubsection*{Right coideal subalgebras of $U_q(\sl_2)$}
Obviously $U_q(\sl_2)$ itself and the smaller version $\langle EK^{-1},K^{2}, K^{-2},F\rangle$ are RCS. The only \emph{homogeneous} RCS are the Standard Borel subalgebras $U^{\geq0}$ and $U^{\leq0}$ and $U^0$.

In $U^{\geq0}$ resp. $U^{\leq0}$ there exist moreover the RCS $\langle EK^{-1}, K^j\rangle$ and $\langle F, K^j\rangle$ for some $j$
and there exist families of character shifted RCS
$\langle EK^{-1}\rangle_{\phi^+}$ resp. $\langle F\rangle_{\phi^-}$ for characters on $\langle EK^{-1}\rangle$ resp. $\langle F\rangle$ given by $\phi^+(EK^{-1})=\lambda$ and $\phi^-(F)=\lambda'$.
They have the form $\langle EK^{-1}\rangle_{\phi^+}=\langle EK^{-1}+\lambda K^{-1}\rangle$ and $\langle F\rangle_{\phi^-}=\langle F+\lambda'K^{-1}\rangle$.
Moreover there exists of course an RCS of the form 
$\langle EK^{-1}+c_FF+c_KK^{-1}\rangle$.

Next to those obvious RCS there exists also a special RCS $B=\langle EK^{-1}+\lambda K^{-1}, F+\lambda'K^{-1}\rangle$ with $\lambda\lambda'=\frac{q^2}{(1-q^2)(q-q^{-1})}$. 
The $q$-commutator of the generators is:
\begin{align}\label{kommutator}
[EK^{-1}+\lambda K^{-1},F+\lambda'K^{-1}]_{q^2}=\frac{q^2}{q-q^{-1}}(1-K^{-2})+(1-q^2)\lambda \lambda' K^{-2}=\frac{q^2}{q-q^{-1}}1
\end{align}
$B$ is closed under multiplication and thus an RCS, it is even isomorphic to the Weylalgebra and thus any finite dimensional irreducible representation is one dimensional.
With the results of this paper we could prove in \cite{Vocke16} that $B$ is even a Borel subalgebra (i.e. an RCS maximal with the properyt that each irreducible finite dimensional representation is one dimensional).\\

Note, that different right coideal subalgebras $B$ for different choices of $\lambda,\lambda'$ can be mapped into each other via the Hopf-Automorphism $E\mapsto tE,\;F\mapsto t^{-1}F$.

\subsubsection*{Classification of right coideal subalgebras}
%Jetzt wählen wir eine beliebige Borelunteralgebren $C\subseteq U=U_q(sl_2)$ und zeigen: Entweder gilt $C\subset B$, also $C=B$ weil C maximal ist mit obiger Eigenschaft (also insbesondere B Borel) oder $U^+\subset C$ also $U^{\geq0}=C$ bzw. $ U^-\subset C$also $U^{\leq0}=C$. 
%Mit der Maximalitätseigenschaft folgt dann die Behauptung.\\

We now prove that the previous list of RCS is complete. We restrict to the case of right coideal subalgebras $C$ whith the property $C^0:=C\cap U^0$ is a sub Hopfalgebra, and use the knowledge obtained in the present paper on generating systems to classify all RCS of $U_q(\sl_2)$:
We consider the set of all possible generating elements,
as in Lemma \ref{leittermwahl}. We know there exists a generating system for each RCS whose elements lie in the set $EZS:=\{K^i,c_EEK^{-1}+c_FF+c_KK^{-1}\}$ for $i\in\mathbf{Z}$ and constants $c_E,c_F,c_K\in k$. Each such generating system of an arbitrary right coideal subalgebra contains at most three elements, that is one with $E$-leading term in the root $\alpha$, one with $F$-leading term in the root $\alpha$ and one in $U^0$.\\

RCS $C$ with $C\cap U^0=\emptyset$: From (\ref{kommutator}) we know, that the only RCS with $C\cap U^0=\emptyset$ which does not lie in $U^{\geq0}$ or in $U^{\leq0}$, is either $\langle EK^{-1}+\lambda F+\lambda' K^{-1}\rangle$ for some $\lambda,\lambda'$ or is of the form $\langle EK^{-1}+\lambda K^{-1}, F+\lambda'K^{-1}\rangle$ with $\lambda\lambda'=\frac{q^2}{(1-q^2)(q-q^{-1})}$.\\

RCS $C$ with $C\cap U^0\neq\emptyset$: There is a $j$ with $K^j\in C$, thus due to Lemma \ref{0part} there can be no character-shifted generating element of $C$ nor a generating element with nontrivial leading terms in $E$ and $F$. Thus $C$ either lies in $U^{\geq0}$ or in $U^{\leq0}$ or contains as well $EK^{-1}$ as $F$ and thus also $K^2$ and $K^{-2}$, then $C$ is either $\langle EK^{-1},K^{2}, K^{-2},F\rangle$ or $U_q(\sl_2)$ itself.

\subsubsection*{On representation theory of RCS of $U_q(\sl_2)$}
In \cite{Vocke16} we used the results of this paper to prove some interesting results on Borel subalgebras and representation theory of RCS in $U_q(\sl_2)$:
\begin{satz}[\cite{Vocke16} Theorem 9.1]
There are two different types of Borel subalgebras of the $U_q(\sl_2)$: the Standard-Borel subalgebras $U^{\geq0}$ and $U^{\leq0}$ 
and the Borel subalgebra $B:=\langle EK^{-1}+\lambda K^{-1}, F+\lambda'K^{-1}\rangle$ with $\lambda\lambda'=\frac{q^2}{(1-q^2)(q-q^{-1})}$.
\end{satz}
In the proof of the maximality we use the knowledge of all RCS in $U_q(\sl_2)$.
\begin{Bsp}
If we restrict the finite dimensional irreducible representations $L(\lambda)$ on $U$ to a Borel subalgebra $B$ the representation is no more irreducible, but not necessary semi simple. In the case of a Borel subalgebra all the composite factors have to be one dimensional.\\

%Konkret sieht das bei $U_q(\sl_2)$ wie folgt aus:in eindimensionale irreduzible Darstellungen zerfallen betrachten wir exemplarisch den Höchstgewichtsmodul.
In $U_q(\sl_2)$: Given a highest weight vector of $L(\lambda)$ to the only dominant weight 
$\lambda=\frac{1}{2}\alpha$ and let $x_0:=v_{\lambda},~x_1:=v_{\lambda-\alpha}$, then the elements in $B$ act as follows:
\begin{align*}
& (EK^{-1}+\lambda_E K^{-1}).x_0=\lambda q^{-1}x_0\quad &(EK^{-1}+\lambda_EK^{-1}).x_1=\lambda qx_1+qx_0\\
& (F+\lambda'K^{-1}).x_0=x_1+\lambda'q^{-1}x_0\quad &(F+\lambda'K^{-1}).x_1=\lambda'qx_1
\end{align*}

One can easily check, that $L(\lambda)$ has a one dimensional submodule $\langle x_0+\lambda(1-q^{-2})x_1\rangle $ to the eigenvalue $\lambda q$ of $EK^{-1}+\lambda_EK^{-1}$
resp. $\lambda' q^{-1}$ of $F+\lambda'K^{-1}$ 
%. $L(\lambda)$ hat den eindimensionalen Untermodul $\langle x_0+\lambda(1-q^{-2})x_1\rangle $ 
and the one dimensional quotient module 
$L(\lambda)/(x_0+\lambda(1-q^{-2})x_1)$ with Eigenvalues $\lambda q^{-1}$ resp. $\lambda' q$ . 
\end{Bsp}
\subsection{Right coideal subalgebras of $U_q(\sl_3)$}

\subsubsection*{Data of $U_q(\sl_3)$}\label{relationenkapitelsl3}

%\subsection*{Data der $U_q(\sl_3)$}
 $U_q(\sl_3)$ is the quantized universal enveloping algebra of the Lie algebra $\sl_3$ with the root system $\Phi=\{\alpha,\beta,\alpha+\beta,-\alpha,-\beta,-\alpha-\beta\}$ and basis $\Pi=\{\alpha,\beta\}$. 
%with scalar product $(\_,\_)$ given by $(\alpha,\alpha)=2=(\beta,\beta)$ and $(\alpha,\beta)=-1$. 
%The root lattice $\Lambda$ is generated? by the fundamental weights $\lambda_1=\frac{2}{3}\alpha+\frac{1}{3}\beta$ and $\lambda_2=\frac{1}{3}\alpha+\frac{2}{3}\beta$. 
%The algebra $U_q(\sl_3)$ is generated by the elements $\{E_{\alpha},F_{\alpha},E_{\beta},F_{\beta},K_{\alpha},K_{\beta},K^{-1}_{\alpha},K^{-1}_{\beta}\}$.
%Its Weylgroup is $W=\{s_{\alpha},s_{\beta},s_{\alpha}s_{\beta},s_{\beta}s_{\alpha},s_{\alpha}s_{\beta}s_{\alpha}\}$ for simple reflections $s_{\alpha}$ and $s_{\beta}$.
%The reduced expressions are unambiguous for all elements of the Weyl group
%except the longest element $w_0=s_{\alpha}s_{\beta}s_{\alpha}$. Here their are two different choices of the reduced expression, these are $s_{\alpha}s_{\beta}s_{\alpha}$ and $s_{\beta}s_{\alpha}s_{\beta}$.
The root vectors for the simple roots are the generating elements $E_{\alpha},E_{\beta}$ and $F_{\alpha},F_{\beta}$ and those for the root $\alpha+\beta$, depending on the underlying reduced representation, are given by:
\begin{align*}
 &E_{\alpha\beta}:=T_{\alpha}(E_{\beta})=E_{\alpha}E_{\beta}-q^{-1}E_{\beta}E_{\alpha}~ &E_{\beta\alpha}:=T_{\beta}(E_{\alpha})=-q^{-1}(E_{\alpha}E_{\beta}-qE_{\beta}E_{\alpha})\\
 &F_{\alpha\beta}:=T_{\alpha}(F_{\beta})=-q(F_{\alpha}F_{\beta}-q^{-1}F_{\beta}F_{\alpha})~    &F_{\beta\alpha}:=T_{\beta}(F_{\alpha})=F_{\alpha}F_{\beta}-qF_{\beta}F_{\alpha}
\end{align*}
The quantum group is symmetric in $\alpha$ and $\beta$ and with the map $\omega$ also symmetric in $E$ and $F$. For this reason, in the following we will only calculate the relations which do not result from symmetry from the others.
The q-commutator ($[x,y]:=xy-qyx$) relations of the root vectors $E_{\alpha\beta}$ and $E_{\beta\alpha}$ are:
\begin{align*}
 &[E_{\alpha\beta},F_{\alpha\beta}]_1=\frac{K_{\alpha+\beta}-K_{\alpha+\beta}^{-1}}{q-q^{-1}}\quad&\quad& \\
 &[E_{\alpha\beta},F_{\alpha}]_1=-E_{\beta}K_{\alpha}^{-1}\quad&[E_{\beta\alpha},F_{\alpha}]_1=q^{-1}E_{\beta}K_{\alpha}\quad&\\
&[E_{\alpha\beta},E_{\alpha}]_{q^{-1}}=0 \quad&[E_{\beta\alpha},E_{\alpha}]_{q}=0 \quad&
\end{align*}
The comultiplication is given by:
$$\Delta(E_{\alpha})=E_{\alpha}\otimes 1+K_{\alpha}\otimes E_{\alpha}\quad \Delta(F_{\alpha})=F_{\alpha}\otimes K_{\alpha}^{-1}+1\otimes F_{\alpha}$$
\begin{align*}
&\Delta(E_{\alpha\beta})=E_{\alpha\beta}\otimes 1+K_{\alpha+\beta}\otimes E_{\alpha\beta}+(1-q^{-2})E_{\alpha}K_{\beta}\otimes E_{\beta}\\
&\Delta(F_{\alpha\beta})=F_{\alpha\beta}\otimes K_{\alpha+\beta}^{-1}+1\otimes F_{\alpha\beta}+(q^{-1}-q)F_{\beta}\otimes F_{\alpha}K_{\beta}^{-1}
\end{align*}

Now we want to classify all right coideal subalgebras of $U_q(\sl_3)$. Therefore we use the knowledge of the underlying paper and 
first specify all possible generating elements of any RCS $C$ in $U_q(\sl_3)$ with the property $C^0:=C\cap U^0$ is a sub Hopfalgebra. In the following we will assume $C$ to be an RCS of $U_q(\sl_3)$ with this property.
\subsubsection*{Possible generating elements of an RCS of $U_q(\sl_3)$}
%Wir erhalten die Notation von oben, allerdings sei 
%$\phi^+$ ein Charakter auf $C^{\geq0}$ und $\phi^-$ ein Charakter auf $C^{\leq0}$. Außerdem seien $c_F,c_K\in k^*$ Konstanten, pro Term beliebig.
From Theorem \ref{hauptsatzgenerator} we know, that we can choose a generating system of $C$, whose elements have a very special form. We now want to specify these possible elements explicitly. In $U_q(\sl_3)$ up to symmetry in $E$ and $F$ resp. $\alpha$ and $\beta$ we can give the complete list of all possible generating elements of $C$:

\begin{List}[possible generating elements]\label{listebasis}\leavevmode\newline
\begin{enumerate}

\item homogeneuos elements in $U^{\geq0}$
\begin{enumerate}
\item elements in $U^0$
 \item root vectors in $U^{\geq0}$: \begin{itemize}
            \item $E_{\alpha}K_{\alpha}^{-1}$
            %,F_{\alpha},E_{\beta}K_{\beta}^{-1},F_{\beta}$
            \item $E_{\alpha\beta}K_{\alpha+\beta}^{-1}$
            %E_{\beta\alpha}K_{\alpha+\beta}^{-1},F_{\alpha\beta},F_{\beta\alpha}$
            \end{itemize}
 \end{enumerate}
 \item character- shifted root vectors in $U^{\geq0}$: \begin{enumerate}
 \item $E_{\alpha\beta}K_{\alpha+\beta}^{-1}+\lambda_{\alpha}(1-q^{-2})E_{\beta}K_{\alpha+\beta}^{-1}$\\
  for $\lambda_{\alpha}\in k^*$
                       \begin{itemize}
                       \item RCS that contain this element also contain $E_{\alpha}K_{\alpha}^{-1}+\lambda_{\alpha}K_{\alpha}^{-1}$
                       \end{itemize}

\item $E_{\alpha\beta}K_{\alpha+\beta}^{-1}+\lambda_{\alpha\beta}K_{\alpha+\beta}^{-1}$\\
 for $\lambda_{\alpha\beta}\in k^*$
                       \begin{itemize}
                       \item RCS that contain this element also contain $E_{\alpha}K_{\alpha}^{-1}$
                       \end{itemize}

\item $E_{\alpha}K_{\alpha}^{-1}+\lambda_{\alpha}K_{\alpha}^{-1}$\\
 for $\lambda_{\alpha}\in k^*$
                       %$(\phi^+\otimes id)\Delta(E_{\alpha})=, (\phi^+\otimes id)\Delta(E_{\beta})=E_{\beta}K_{\beta}^{-1}+\lambda_{\beta}K_{\beta}^{-1}$
                       
                       %$(\phi^+\otimes id)\Delta(E_{\alpha\beta})=, (\phi^+\otimes id)\Delta(E_{\beta\alpha})E_{\beta\alpha}K_{\alpha+\beta}^{-1}+\lambda_{\beta\alpha}K_{\alpha+\beta}^{-1}$
                       
                       %(\phi^+\otimes id)\Delta(E_{\alpha\beta})=und $E_{\beta\alpha}K_{\alpha+\beta}^{-1}+\lambda_{\beta}(1-q^{-2})E_{\alpha}K_{\alpha+\beta}^{-1}$, ebenso in $U^{\leq0}$.
                       \end{enumerate}
                       %\end{enumerate}

 \item elements in $U$ with non trivial F- and E-leading term 
 \begin{enumerate}
 
\item $E_{\alpha\beta}K_{\alpha+\beta}^{-1}+\lambda_{\alpha}^+(1-q^{-2})E_{\beta}K_{\alpha+\beta}^{-1}+c_FF_{\beta\alpha}+c_F(q^{-1}-q)\lambda_{\alpha}^-F_{\beta}K_{\alpha}^{-1}$\\ 
for $c_F\in k^*$ and $\lambda_{\alpha}^+,\lambda_{\alpha}^-\in k$ with $\lambda_{\alpha}^-\lambda_{\alpha}^+=\frac{q^2}{(1-q^2)(q-q^{-1})}$
\begin{itemize}
     \item RCS that contain this element also contain $E_{\alpha}K_{\alpha}^{-1}+\lambda_{\alpha}^+K_{\alpha}^{-1}$ and $F_{\alpha}+\lambda_{\alpha}^-K_{\alpha}^{-1}$
\end{itemize}

\item $E_{\alpha\beta}K_{\alpha+\beta}^{-1}+\lambda_{\alpha}(1-q^{-2})E_{\beta}K_{\alpha+\beta}^{-1}+c_FF_{\alpha\beta}+c_F\lambda_{\beta}^-(q^{-1}-q)F_{\alpha}K_{\beta}^{-1}+c_KK_{\alpha+\beta}^{-1}$\\
for $c_F\in k^*$ and $\lambda_{\alpha}^+,\lambda_{\beta}^-,c_K\in k$ with $\lambda_{\alpha}^+\lambda_{\beta}^-=\lambda_{\beta}^-c_K=\lambda_{\alpha}^+c_K=0$
\begin{itemize}
     \item RCS that contain this element also contain $E_{\alpha}K_{\alpha}^{-1}+\lambda_{\alpha}^+K_{\alpha}^{-1}$ and 	 $F_{\beta}+\lambda_{\beta}^-K_{\beta}^{-1}$
\end{itemize}

\item $E_{\alpha\beta}K_{\alpha+\beta}^{-1}+c_{\alpha}(1-q^{-2})q^{-1} F_{\alpha}E_{\beta}K_{\beta}^{-1}-c_{\alpha}c_{\beta}q^{-2}F_{\alpha\beta}+(1-q^{-2})q^{-1}c_{\beta}'c_{\alpha}F_{\alpha}K_{\beta}^{-1}+c_{\alpha}'(1-q^{-2})E_{\beta}K_{\alpha+\beta}^{-1}+c_{\alpha}'c_{\beta}'(1-q^{-1})q^{-1}K_{\alpha+\beta}^{-1}$\\
for $c_{\alpha},c_{\beta}\in k^*$ and $c_{\alpha}',c_{\beta}'\in k$ are either $0$ or determined
%???
 by $c_{\alpha},c_{\beta}$
\begin{itemize}
     \item RCS that contain this element also contain $E_{\alpha}K_{\alpha}^{-1}+c_{\alpha}F_{\alpha}+c_{\alpha}'K_{\alpha}^{-1}$ and $E_{\beta}K_{\beta}^{-1}+c_{\beta}F_{\beta}+c_{\beta}'K_{\beta}^{-1}$
\end{itemize} 
 
\item $E_{\beta\alpha}-(1-q^2)c_{\alpha}E_{\alpha}F_{\alpha}K_{\beta}^{-1}-qc_{\alpha}c_{\beta}F_{\alpha\beta}+ \frac{qc_{\alpha}-c_{\beta}}{1-q^{-2}}K_{\alpha+\beta}^{-1}$\\
 for $c_{\alpha},c_{\beta}\in k^*$
\begin{itemize}
\item RCS that contain this element also cointain $E_{\beta}K_{\alpha}^{-1}+c_{\alpha}F_{\alpha}$ and $E_{\alpha}K_{\beta}^{-1}+c_{\beta}F_{\beta}$ and $K_{\alpha}K_{\beta}^{-1}$
\end{itemize}

 \item $E_{\alpha}K_{\beta}^{-1}+c_F F_{\beta}$\\
 for $c_F\in k^*$
               \begin{itemize}
                       \item RCS that contain this element also contain $K_{\beta}^{-1}K_{\alpha}$
                       \end{itemize}
               %, $E_{\beta}K_{\beta}^{-1}+\lambda_FF_{\alpha}$

 \item $E_{\alpha}K_{\alpha+\beta}^{-1}+c_FF_{\alpha\beta}$\\
 for $c_F\in k^*$
               \begin{itemize}
                       \item RCS that contain this element also contain $K_{\beta}^{-1}$, $F_{\beta}$
                       \end{itemize}
               %, $E_{\alpha}K_{\alpha}^{-1}+\lambda_FF_{\beta\alpha}$, $E_{\beta}K_{\beta}^{-1}+\lambda_FF_{\alpha\beta}$, $E_{\beta}K_{\beta}^{-1}+\lambda_FF_{\beta\alpha}$
               %\item $F_{\alpha}+\lambda_EE_{\alpha\beta}K_{\alpha+\beta}^{-1}$, $F_{\alpha}+\lambda_EE_{\beta\alpha}K_{\alpha+\beta}^{-1}$, $F_{\beta}+\lambda_EE_{\alpha\beta}K_{\alpha+\beta}^{-1}$, $F_{\beta}+\lambda_EE_{\beta\alpha}K_{\alpha+\beta}^{-1}$
%               \end{itemize}

 \item $E_{\alpha}K_{\alpha}^{-1}+c_FF_{\alpha}+c_KK_{\alpha}^{-1}$\\
 for $c_F\in k^*$ and $c_K\in k$

     \end{enumerate}
%ALTE VERSION:
% \item $(\phi^+\otimes id)\Delta(E_{\alpha})+c_F(\phi^-\otimes id)\Delta(F_{\alpha})+c_KK_{\alpha}^{-1}$
% %, $(\phi^+\otimes id)\Delta(E_{\beta})+\lambda_F(\phi^-\otimes id)\Delta(F_{\beta})$
% \item $(\phi^+\otimes id)\Delta(E_{\alpha\beta})+c_F(\phi^-\otimes id)\Delta(F_{\alpha\beta})+c_KK_{\alpha+\beta}^{-1}$
% %, $(\phi^+\otimes id)\Delta(E_{\alpha\beta})+\lambda_F(\phi^-\otimes id)\Delta(F_{\beta\alpha})$,
% %$(\phi^+\otimes id)\Delta(E_{\beta\alpha})+\lambda_F(\phi^-\otimes id)\Delta(F_{\alpha\beta})$ und $(\phi^+\otimes id)\Delta(E_{\beta\alpha})+\lambda_F(\phi^-\otimes id)\Delta(F_{\beta\alpha})$
%\end{itemize}
 \end{enumerate}
 \end{List}

\subsubsection*{Right coideal subalgebras of $U_q(\sl_3)$}\label{kapitelrcssl3}
%Nun wollen wir unser Wissen über Rechtscoidealunteralgebren von $U_q(\mathfrak{g})$ aus Kapitel \ref{bisherigeklassifikation} auf die $U_q(\sl_3)$ anwenden.
We now want to give a complete list of all RCS of $U_q(\sl_3)$ with the property $C^0:=C\cap U^0$ is a sub Hopfalgebra. Therefore we will first specify the homogeneous RCS ($U^0\subset C$) in the positive Borel Part and give a classification of homogeneous RCS, as done in \cite{HK11a}. This gives us easily a list of all homogeneous generated RCS of $U_q(\sl_3)$. 
Then we will use \cite{HK11b} for explicitly specifying all RCS in $U_q(\sl_3^{\geq0})$. We then give a complete list of all RCS which are generated by these socalled character- shifted elements (i.e. generated by elements of type \ref{listebasis} 2)).

Lastly, we use the results of this paper to consider general RCS by giving a list of all RCS which are generated by at least one element of type \ref{listebasis} 3).

\paragraph*{Homogeneous RCS of the positive Borel part of $U_q(\sl_3)$}
In the positive Borelpart of the $U_q(\sl_3)$ lie up to isomorphism exactly six types of homogeneous RCS, i.e. one per Weyl group element:
\begin{itemize}
\item $U^0$
\item $ U^+[s_{\alpha}]U^0=\langle E_{\alpha},U^0\rangle$
\item $ U^+[s_{\beta}]U^0=\langle E_{\beta},U^0\rangle$
\item $ U^+[s_{\alpha}s_{\beta}]U^0=\langle E_{\alpha},E_{\alpha\beta},U^0\rangle$
\item $ U^+[s_{\beta}s_{\alpha}]U^0=\langle E_{\beta},E_{\beta\alpha},U^0\rangle$ 
\item $U^{\geq0}$
\end{itemize}
 The RCS in $U^{\leq0}$ are constructed analogously. If we want to construct homogeneous generated RCS which are not necessary homogeneous (in the sense $U^0\subset C$), we can exchange $U^+[s]$ by $\psi(U^+[s])$ and $U^0$ by $T_L$ for some subgroup $L\subset \Phi$.
 
\paragraph*{Homogeneous RCS of $U_q(\sl_3)$}
From \cite{HK11b} we know all homogeneous RCS in $U$. If we consider RCS generated by homogeneous elements we get up to isomorphism and symmetry the following list of remaining homogeneous RCS: 
\begin{itemize}
\item $\psi(U^+[s_{\alpha}])T_LU^-[s_{\alpha}]$ for $K_{\alpha}^2\subset L$
\item $\psi(U^+[s_{\alpha}s_{\beta}])T_LU^-[s_{\alpha}]$ for $K_{\alpha}^2\subset L$
\item $\psi(U^+[s_{\beta}])T_LU^-[s_{\alpha}]$ 
\item $\psi(U^+[s_{\beta}s_{\alpha}])T_LU^-[s_{\alpha}]$ for $K_{\alpha}^2\subset L$
\item $\psi(U^+)T_LU^-[s_{\alpha}]$
 %$\psi(U^+U^0U^-[s_{\alpha}s_{\beta}]$, 
\item $\psi(U^+[s_{\alpha}s_{\beta}])T_LU^-[s_{\alpha}s_{\beta}]$ for $K_{\alpha}^2, K_{\beta}^2\subset L$
\item $\psi(U^+[s_{\beta}s_{\alpha}])T_LU^-[s_{\alpha}s_{\beta}]$ for $K_{\alpha+\beta}^2\subset L$
\item $\psi(U^+)T_LU^-[s_{\alpha}s_{\beta}]$ for $K_{\alpha}^2, K_{\beta}^2\subset L$
\item $\psi(U^+)T_LU^-$ for $K_{\alpha}^2, K_{\beta}^2\subset L$
\end{itemize}

 \paragraph*{RCS in the positive Borelpart of $U_q(\sl_3)$}
 The RCS $C$ in $U^{\geq0}$ are by \cite{HS09} generated via character shifting. 
We give here explicitly up to isomorphism all types of non homogeneous RCS in the positive Borelpart of $ U_q(\sl_3) $.
%Analog definieren wir $\phi_{\beta}$ und $\phi_{\alpha\beta}$, bzw. $\phi_{\beta\alpha}$. 
Consider first the connected RCS in $C$ (i.e. those for which $U^0\cap C=k\cdot1$).\\

There is the RCS $ \langle E_{\alpha}K_{\alpha}^{-1}\rangle_{\phi_{\alpha}}$ with character $\phi_{\alpha}$ defied by 
$\phi_{\alpha}(E_\alpha K_{\alpha}^{-1})=\lambda_{\alpha}$.
%auf $\langle E_{\alpha}K_{\alpha}^{-1}\rangle$
Moreover there is the RCS $\langle E_{\alpha}K_{\alpha}^{-1},E_{\alpha\beta}K_{\alpha+\beta}^{-1}\rangle_{\phi_{\alpha}}$ 
 with character defined by
 %$\langle E_{\alpha}K_{\alpha}^{-1},E_{\alpha\beta}K_{\alpha+\beta}^{-1}\rangle$ 
 $\phi_{\alpha}(E_\alpha K_{\alpha}^{-1})=\lambda_{\alpha}$ and $\phi_{\alpha}(E_{\alpha\beta}K_{\alpha+\beta}^{-1})=0$, 
%, analog für $\beta$ statt $\alpha$.
moreover there is the RCS $ \langle E_{\alpha}K_{\alpha}^{-1},E_{\alpha\beta}K_{\alpha+\beta}^{-1}\rangle_{\phi_{\alpha\beta}}$ for a character $\phi_{\alpha\beta}$ on $\langle E_{\alpha}K_{\alpha}^{-1},E_{\alpha\beta}K_{\alpha+\beta}^{-1}\rangle$
with $\phi_{\alpha\beta}(E_{\alpha}K_{\alpha}^{-1})=0$ and $\phi_{\alpha\beta}(E_{\alpha\beta}K_{\alpha+\beta}^{-1})=\lambda_{\alpha\beta}$. \\ 
%ELEMENTE GENAUER 2017
%These right coideal subalgebras are generated by character-shifted elements, which we want to specify explicitly. Let
% $c_{\alpha}:=\lambda_{\alpha}(1-q^{-2})$ and $c_{\beta}:=\lambda_{\beta}(1-q^{-2})$ and $\phi_{\nu}(E_{\mu}K_{\mu}^{-1})=\delta_{\mu\nu}\lambda_{\nu}$ for $\mu,\nu\in\Phi^+$.
%
%\begin{center}
%\begin{tabular}{|c||c|c|c|c|}
%\hline
% &$E_{\alpha}K_{\alpha}^{-1}$&$E_{\beta}K_{\beta}^{-1}$&$E_{\alpha\beta}K_{\alpha+\beta}^{-1}$&$E_{\beta\alpha}K_{\alpha+\beta}^{-1}$\\
% \hline\hline
% $(\phi_{\alpha}\otimes id)\Delta$&$E_{\alpha}K_{\alpha}^{-1}+\lambda_{\alpha}K_{\alpha}^{-1}$&$E_{\beta}K_{\beta}^{-1}$&$E_{\alpha\beta}K_{\alpha+\beta}^{-1}+c_{\alpha}E_{\beta}K_{\alpha+\beta}^{-1}$&$E_{\beta\alpha}K_{\alpha+\beta}^{-1}$\\
% \hline
%  $(\phi_{\beta}\otimes id)\Delta$&$E_{\alpha}K_{\alpha}^{-1}$&$E_{\beta}K_{\beta}^{-1}+\lambda_{\beta}K_{\beta}^{-1}$&$E_{\alpha\beta}K_{\alpha+\beta}^{-1}$&$E_{\beta\alpha}K_{\alpha+\beta}^{-1}+c_{\beta}E_{\alpha}K_{\alpha+\beta}^{-1}$\\
%\hline
%  $(\phi_{\alpha\beta}\otimes id)\Delta$&$E_{\alpha}K_{\alpha}^{-1}$&$E_{\beta}K_{\beta}^{-1}$&$E_{\alpha\beta}K_{\alpha+\beta}^{-1}+\lambda_{\alpha\beta}K_{\alpha+\beta}^{-1}$&$E_{\beta\alpha}K_{\alpha+\beta}^{-1}+\lambda_{\alpha\beta}K_{\alpha+\beta}^{-1}$\\
%  \hline
%  \end{tabular}
%\end{center}
All RCS $C$ in $U_q(\sl_3)^{\geq0}$ have the form $C=C^+T_L$ for a $C^+$ as above. 
 Due to \cite{HK11b} the following restriction holds for $T_L$: 
$L \subset \supp(\phi)^{\bot}$. Thus in the cases above $T_L$ can have the following form: In the case of $\phi_{\alpha}$ $T_L$ lies in $\langle K_{2\beta+\alpha},K_{2\beta+\alpha}^{-1}\rangle$.
In the case of $\phi_{\alpha+\beta}$ $T_L$ lies in $\langle K_{\alpha-\beta},K_{\alpha-\beta}^{-1}\rangle$. 
%Für $\beta$ statt $\alpha$ analog. 
%das heißt für die Rechtscoidealunteralgebren oben, die um $\phi_{\alpha}$ verschoben sind besteht 
%$T_L$ aus Elementen $K_{2\beta+\alpha}^i$ mit $i\in \mathbb{Z}$, für $\beta$ analog und für diese die mit $\phi_{\alpha\beta}$ verschoben sind 
%kann $T_L$ von Elementen der Form $K_{\alpha-\beta}^i$ für ein $i\in \mathbb{Z}$ erzeugt werden. 
The character-shifted RCS of the $ U_q(\sl_3)^ {\leq0} $ are constructed similarly.
 %\end{align}
 % (\phi_{\alpha}\otimes id)\Delta(E_{\alpha}K_{\alpha}^{-1})&=E_{\alpha}K_{\alpha}^{-1}+\lambda_{\alpha}K_{\alpha}^{-1} (\phi_{\alpha}\otimes id)\Delta(E_{\beta}K_{\beta}^{-1})&=E_{\beta}K_{\beta}^{-1} (\phi_{\alpha}\otimes id)\Delta(E_{\alpha\beta}K_{\alpha+\beta}^{-1})&=E_{\alpha\beta}K_{\alpha+\beta}^{-1}\\
 %(\phi_{\beta}\otimes id)\Delta(E_{\alpha}K_{\alpha}^{-1})&=E_{\alpha}K_{\alpha}^{-1} (\phi_{\beta}\otimes id)\Delta(E_{\beta}K_{\beta}^{-1})&=E_{\beta}K_{\beta}^{-1}+\lambda_{\beta}K_{\beta}^{-1} (\phi_{\beta}\otimes id)\Delta(E_{\alpha\beta}K_{\alpha+\beta}^{-1})&=E_{\alpha\beta}K_{\alpha+\beta}^{-1}\\
 %(\phi_{\alpha\beta}\otimes id)\Delta(E_{\alpha}K_{\alpha}^{-1})&=E_{\alpha}K_{\alpha}^{-1} (\phi_{\alpha\beta}\otimes id)\Delta(E_{\beta}K_{\beta}^{-1})&=E_{\beta}K_{\beta}^{-1} (\phi_{\alpha\beta}\otimes id)\Delta(E_{\alpha\beta}K_{\alpha+\beta}^{-1})&=E_{\alpha\beta}K_{\alpha+\beta}^{-1}+\lambda_{\alpha\beta}K_{\alpha+\beta}^{-1}
 %\end{align}

To give now a complete list of RCS in $U_q(\sl_3)$ we consider step by step RCS with generating elements in the list \ref{listebasis}.
In particular we list to each generating element all RCS up to symmetry in E and F resp. in $\alpha$ and $\beta$ which contain this element in their generating set, such that it cannot be written as a sum of other generating elements.\\

As we already know all homogeneous RCS, we can restrict to the case of an RCS $C$ which is not generated by only homogeneous elements, thus it contains at least one generating element from point 2) or 3) in \ref{listebasis}. First we consider those RCS which are generated only by generating elements from point 1) or 2), thus every generating element lies either in $U^{\geq0}$ or in $U^{\leq0}$. We list all possible RCS containing an element of 2) step by step up to symmetry in E and F resp. $\alpha$ and $\beta$.

\begin{itemize}
\item[2a)] For an arbitrary $0\neq \lambda_{\alpha},\lambda_{\alpha}^-\in k*$ with $\lambda_{\alpha}\lambda_{\alpha}^-==\frac{q^2}{(1-q^2)(q-q^{-1})}$ and $i\in \mathbb{N}$ there exist the following RCS :
\begin{itemize}
\item $\langle E_{\alpha\beta}K_{\alpha+\beta}^{-1}+\lambda_{\alpha}(1-q^{-2})E_{\beta}K_{\alpha+\beta}^{-1},E_{\alpha}K_{\alpha}^{-1}+\lambda_{\alpha}K_{\alpha}^{-1}, (K_{\alpha}K_{\beta}^2)^i\rangle$
\item $\langle E_{\alpha\beta}K_{\alpha+\beta}^{-1}+\lambda_{\alpha}(1-q^{-2})E_{\beta}K_{\alpha+\beta}^{-1},E_{\alpha}K_{\alpha}^{-1}+\lambda_{\alpha}K_{\alpha}^{-1}, K_{\alpha}K_{\beta}^2, F_{\beta}\rangle$
\item $\langle E_{\alpha\beta}K_{\alpha+\beta}^{-1}+\lambda_{\alpha}(1-q^{-2})E_{\beta}K_{\alpha+\beta}^{-1},E_{\alpha}K_{\alpha}^{-1}+\lambda_{\alpha}K_{\alpha}^{-1}, K_{\alpha}K_{\beta}^2, F_{\beta}, F_{\alpha\beta}\rangle$
\item $\langle E_{\alpha\beta}K_{\alpha+\beta}^{-1}+\lambda_{\alpha}(1-q^{-2})E_{\beta}K_{\alpha+\beta}^{-1},E_{\alpha}K_{\alpha}^{-1}+\lambda_{\alpha}K_{\alpha}^{-1}, (K_{\alpha}K_{\beta}^2)^i, F_{\alpha}+\lambda_{\alpha}^-K_{\alpha}^{-1}\rangle$
\item $\langle E_{\alpha\beta}K_{\alpha+\beta}^{-1}+\lambda_{\alpha}(1-q^{-2})E_{\beta}K_{\alpha+\beta}^{-1},E_{\alpha}K_{\alpha}^{-1}+\lambda_{\alpha}K_{\alpha}^{-1}, (K_{\alpha}K_{\beta}^2)^i, F_{\alpha}+\lambda_{\alpha}^-K_{\alpha}^{-1}, F_{\beta\alpha}+(q-q^{-1})\lambda_{\alpha}^-F_{\beta}K_{\alpha}^{-1}\rangle$
\item $\langle E_{\alpha\beta}K_{\alpha+\beta}^{-1}+\lambda_{\alpha}(1-q^{-2})E_{\beta}K_{\alpha+\beta}^{-1},E_{\alpha}K_{\alpha}^{-1}+\lambda_{\alpha}K_{\alpha}^{-1}, (K_{\alpha}K_{\beta}^2)^i, E_{\beta}K_{\beta}^{-1}\rangle$
\item $\langle E_{\alpha\beta}K_{\alpha+\beta}^{-1}+\lambda_{\alpha}(1-q^{-2})E_{\beta}K_{\alpha+\beta}^{-1},E_{\alpha}K_{\alpha}^{-1}+\lambda_{\alpha}K_{\alpha}^{-1}, (K_{\alpha}K_{\beta}^2)^i, F_{\alpha}+\lambda_{\alpha}^-K_{\alpha}^{-1}, E_{\beta}K_{\beta}^{-1}\rangle$
\item $\langle E_{\alpha\beta}K_{\alpha+\beta}^{-1}+\lambda_{\alpha}(1-q^{-2})E_{\beta}K_{\alpha+\beta}^{-1},E_{\alpha}K_{\alpha}^{-1}+\lambda_{\alpha}K_{\alpha}^{-1}, K_{\alpha}K_{\beta}^2, F_{\alpha}+\lambda_{\alpha}^-K_{\alpha}^{-1}, E_{\beta}K_{\beta}^{-1}, F_{\beta\alpha}+(q-q^{-1})\lambda_{\alpha}^-F_{\beta}K_{\alpha}^{-1}\rangle$
\end{itemize}
\item[2b)] For an arbitrary $0\neq\lambda_{\alpha\beta}\lambda_{\alpha\beta}^-\in k*$ with $\lambda_{\alpha\beta}\lambda_{\alpha\beta}^-==\frac{q^2}{(1-q^2)(q-q^{-1})}$ and $i\in \mathbb{N}$ there exist the following RCS:
\begin{itemize}
\item $\langle E_{\alpha\beta}K_{\alpha+\beta}^{-1}+\lambda_{\alpha\beta}K_{\alpha+\beta}^{-1},E_{\alpha}K_{\alpha}^{-1}, (K_{\alpha}K_{\beta}^{-1})^i\rangle$
\item $\langle E_{\alpha\beta}K_{\alpha+\beta}^{-1}+\lambda_{\alpha\beta}K_{\alpha+\beta}^{-1},E_{\alpha}K_{\alpha}^{-1}, (K_{\alpha}K_{\beta}^{-1})^i, F_{\beta}\rangle$
\item $\langle E_{\alpha\beta}K_{\alpha+\beta}^{-1}+\lambda_{\alpha\beta}K_{\alpha+\beta}^{-1},E_{\alpha}K_{\alpha}^{-1}, (K_{\alpha}K_{\beta}^{-1})^i, F_{\beta}, F_{\alpha\beta}+ \lambda_{\alpha\beta}^-K_{\alpha+\beta}^{-1}\rangle$
\end{itemize}
\item[2c)] Of course there exist the RCS in 2a). For an arbitrary $0\neq \lambda_{\alpha},\lambda_{\alpha}^-\in k*$ with $\lambda_{\alpha}\lambda_{\alpha}^-==\frac{q^2}{(1-q^2)(q-q^{-1})}$ and $i\in \mathbb{N}$ there exist moreover the following RCS :
\begin{itemize}
\item $\langle E_{\alpha}K_{\alpha}^{-1}+\lambda_{\alpha}K_{\alpha}^{-1}, (K_{\alpha}K_{\beta}^2)^i\rangle$
\item $\langle E_{\alpha}K_{\alpha}^{-1}+\lambda_{\alpha}K_{\alpha}^{-1}, (K_{\alpha}K_{\beta}^2)^i, F_{\beta}\rangle$
\item $\langle E_{\alpha}K_{\alpha}^{-1}+\lambda_{\alpha}K_{\alpha}^{-1}, (K_{\alpha}K_{\beta}^2)^i, F_{\beta}, F_{\alpha\beta}\rangle$
\item $\langle E_{\alpha}K_{\alpha}^{-1}+\lambda_{\alpha}K_{\alpha}^{-1}, (K_{\alpha}K_{\beta}^2)^i, F_{\alpha}+\lambda_{\alpha}^-K_{\alpha}^{-1}\rangle$
\end{itemize}
\end{itemize}

Now we give a complete list of all RCS which contain at least one element from \ref{listebasis} 3) in their generating set, such that it cannot be written as a sum of other generating elements:

\begin{itemize}
\item[3a)] For an arbitrary $0\neq \lambda_{\alpha}^+,\lambda_{\alpha}^-\in k*$ with $\lambda_{\alpha}^+\lambda_{\alpha}^-==\frac{q^2}{(1-q^2)(q-q^{-1})}$ there exist the following RCS : $\langle E_{\alpha\beta}K_{\alpha+\beta}^{-1}+\lambda_{\alpha}^+(1-q^{-2})E_{\beta}K_{\alpha+\beta}^{-1}+c_FF_{\beta\alpha}+c_F(q^{-1}-q)\lambda_{\alpha}^-F_{\beta}K_{\alpha}^{-1}, 
E_{\alpha}K_{\alpha}^{-1}+\lambda_{\alpha}^+K_{\alpha}^{-1},F_{\alpha}+\lambda_{\alpha}^-K_{\alpha}^{-1}\rangle$

\item[3b)] For an arbitrary $c_F\neq0$ and $i\in\mathbb{N}$ there are two options: 
\begin{itemize}
\item $\langle E_{\alpha}, F_{\beta}, E_{\alpha\beta}K_{\alpha+\beta}^{-1}+c_FF_{\alpha\beta}, (K_{\alpha}K_{\beta}^{-1})^i\rangle$
\item Let w.l.o.g. $\lambda_{\beta}^-=0$ and choose $\lambda_{\alpha}^+\in k$ and $c_K\in k$ such that $\lambda_{\alpha}^+c_K=0$
then there is another RCS given by\\
 $\langle E_{\alpha\beta}K_{\alpha+\beta}^{-1}+\lambda_{\alpha}(1-q^{-2})E_{\beta}K_{\alpha+\beta}^{-1}+c_FF_{\alpha\beta}+c_KK_{\alpha+\beta}^{-1},E_{\alpha}K_{\alpha}^{-1}+\lambda_{\alpha}^+K_{\alpha}^{-1}, 	 F_{\beta}\rangle$
\end{itemize}

\item[3c)] For $0\neq c_{\alpha},c_{\beta}\in k^*$ and suitably chosen 
%??!! 
$c_{\alpha}',c_{\beta}'$ there is only one possible RCS given by\\
$\langle 
qE_{\alpha\beta}K_{\alpha+\beta}^{-1}+c_{\alpha}(1-q^{-2}) F_{\alpha}E_{\beta}K_{\beta}^{-1}-c_{\alpha}c_{\beta}q^{-1}F_{\alpha\beta}+(1-q^{-2})c_{\beta}'c_{\alpha}F_{\alpha}K_{\beta}^{-1}+c_{\alpha}'(q-q^{-1})E_{\beta}K_{\alpha+\beta}^{-1}+c_{\alpha'}c_{\beta'}(1-q^{-1})K_{\alpha+\beta}^{-1},E_{\alpha}K_{\alpha}^{-1}+c_{\alpha}F_{\alpha}+c_{\alpha}'K_{\alpha}^{-1},E_{\beta}K_{\beta}^{-1}+c_{\beta}F_{\beta}+c_{\beta}'K_{\beta}^{-1}
\rangle$

\item[3d)] For an arbitrary $0\neq c_{\alpha},c_{\beta}\in k*$ there is only one possible RCS given by\\ 
$\langle (1-q^{-2})c_{\alpha}E_{\alpha}F_{\alpha}K_{\beta}^{-1}+q^{-2}E_{\beta\alpha}-q^{-1}c_{\alpha}c_{\beta}F_{\alpha\beta}+ \frac{c_{\alpha}-q^{-1}c_{\beta}}{q-q^{-1}}K_{\alpha+\beta}^{-1},E_{\beta}K_{\alpha}^{-1}+c_{\alpha}F_{\alpha},E_{\alpha}K_{\beta}^{-1}+c_{\beta}F_{\beta},K_{\alpha}K_{\beta}^{-1}
\rangle$

\item[3e)] Of course there is the RCS of (3d)). For an arbitrary $c_F\neq0$ there are three more types of RCS:
\begin{itemize}
\item $\langle E_{\alpha}K_{\beta}^{-1}+c_F F_{\beta},K_{\beta}^{-1}K_{\alpha}\rangle$
\item $\langle E_{\alpha}K_{\beta}^{-1}+c_F F_{\beta},K_{\beta}^{-1}K_{\alpha}, E_{\beta}K_{\beta}^{-1}, E_{\beta\alpha}K_{\alpha+\beta}^{-1}-\frac{q}{q-q^{-1}}c_FK_{\alpha+\beta}^{-1}\rangle$
\item $\langle E_{\alpha}K_{\beta}^{-1}+c_F F_{\beta},K_{\beta}^{-1}K_{\alpha}, E_{\beta}K_{\beta}^{-1}, E_{\beta\alpha}K_{\alpha+\beta}^{-1}-\frac{q}{q-q^{-1}}c_FK_{\alpha+\beta}^{-1}, F_{\alpha},F_{\beta\alpha}+\frac{1}{c_F(q-q^{-1})}K_{\alpha+\beta}^{-1}\rangle$
\end{itemize}

\item[3f)] For an arbitrary $0\neq c_F\in k^*$ there are two RCS
\begin{itemize}
\item $\langle E_{\alpha}K_{\alpha+\beta}^{-1}+c_FF_{\beta\alpha},K_{\beta}^{-1},F_{\beta}\rangle$
\item $\langle E_{\alpha}K_{\alpha+\beta}^{-1}+c_FF_{\beta\alpha},K_{\beta}^{-1},F_{\beta}, E_{\beta}, E_{\beta\alpha}K_{\alpha}^{-1}-c_FF_{\alpha}\rangle$
\end{itemize}

\item[3g)] Of course there exist the RCS from 3c). For an arbitrary $0\neq c_F\in k^*$ and $i\in \mathbb{N}$ there are moreover the following RCS:
\begin{itemize}
\item $\langle 
E_{\alpha}K_{\alpha}^{-1}+c_{\alpha}F_{\alpha}+c_{\alpha}'K_{\alpha}^{-1}(K_{\alpha}K_{\beta}^2)^i
\rangle$
\item $\langle 
E_{\alpha}K_{\alpha}^{-1}+c_{\alpha}F_{\alpha}+c_{\alpha}'K_{\alpha}^{-1},E_{\beta}K_{\beta}^{-1}, E_{\beta\alpha}K_{\alpha+\beta}^{-1},(K_{\alpha}K_{\beta}^2)^i
\rangle$
\end{itemize}

\end{itemize}
%Mixed elements with leadingt terms in different roots: Here due to the $ad T_L$ stable choice 
% %und weil $(\supp(\phi^+)\cap\Pi)\bot (\supp(\phi^-)\cap\Pi)$ und $L\subset(\supp(\phi^+)\cup \supp(\phi^-))^{\bot}$ 
% there are up to isomorphism only the following possibilities:
%The set of these possible generators is called EZS.
%Each right coideal subalgebra can be generated from a selection of these elements, and each of these elements can be contained in a generating system as in ??.
%However, combinations of the elements are only possible to a limited extent.
%
%\paragraph*{Triangular RCS of $U_q(\sl_3)$}
%
%From the Definition of triangular RCS \ref{deftriang} together with the restriction(??) and Theorem \ref{rcsinu+} we get, that a triangular RCS $C$ is always of the form $\psi(U^+[w^+])_{\phi^+}T_LU^-[w^-]_{\phi^-}$, for Weylgroup elements $w^+,w^-\in W$,
%characters $\phi^+$ on $\psi(U^+[w^+])$, $\phi^-$ on $U^-[w^-]$ and a Hopf subalgebra $T_L\subset U^0$ with $L\subset (\supp(\phi^+)\cup \supp(\phi^-))^{\bot}$.\\ 

As an application of the underlying paper we could classify in \cite{Vocke16} all Borel subalgebras (maximal RCS with the property that each simple finite dimensional representation is one dimensional) of $U_q(\sl_3)$ with the knowledge of all RCS in $U_q(\sl_3)$. Next we give the construction of the three types of Borel subalgebras which appear to be all possibile Borel subalgebras due to the classificational result from \cite{Vocke16}.\\

\paragraph*{Type 1: Standard Borel subalgebras}
$U^{\geq0}$ and $U^{\leq0}$ are the socalled standard Borel subalgebras.\\

\paragraph*{Type 2: RCS with a non degeneracy property}
The Borel subalgebras $\psi(U^+[w^+])_{\phi^+}T_LU^-[w^-]_{\phi^-}$ with $\Phi^+(w^+)\cap\Phi^+(w^-)= \supp(\phi^+)\cap \supp(\phi^-)$ 
are up to symmetry isomorphic as algebra to $\psi(U^+[w_0])_{\phi^+}\langle K_{2\beta+\alpha},K_{2\beta+\alpha}^{-1}\rangle U^-[s_{\alpha}]_{\phi^-}$
with $\phi^+(E_{\alpha}K_{\alpha}^{-1})=\lambda_{\alpha}^+$ and $0$ otherwise, as well $\phi^-(F_{\alpha})=\lambda_{\alpha}^-$, such that $\lambda_{\alpha}^+\lambda_{\alpha}^-=\frac{q^2}{(1-q^2)(q-q^{-1})}$.
More precisely there are up to symmetry exactly two such Borel subalgebras. These are 
$$\psi(U^+[w_0])_{\phi^+}\langle K_{2\beta+\alpha},K_{2\beta+\alpha}^{-1}\rangle U^-[s_{\alpha}]_{\phi^-}$$
with characters as above and
$$\psi(U^+[s_{\alpha}s_{\beta}])_{\phi^+}\langle K_{\alpha-\beta}, K_{\alpha-\beta}^{-1}\rangle U^-[s_{\beta}s_{\alpha}]_{\phi^-}$$
with $\phi^+(E_{\alpha\beta}K_{\alpha+\beta}^{-1})=\lambda_{\alpha\beta}^+$ 
also $\phi^-(F_{\alpha\beta})=\lambda_{\alpha\beta}^-$, such that $\lambda_{\alpha\beta}^+\lambda_{\alpha\beta}^-=\frac{q^2}{(1-q^2)(q-q^{-1})}$\\

\paragraph*{Type 3: The RCS $\psi(U^+[s_{\alpha}s_{\beta}])_{\phi^+}\langle K_{2\beta+\alpha},K_{2\beta+\alpha}^{-1}\rangle U^-[s_{\alpha}s_{\beta}]_{\phi^-}$}\label{typ3insl3}

The third type of triangular Borel subalgebras in $U_q(\sl_3)$ is of the form 
$$\psi(U^+[s_{\alpha}s_{\beta}])_{\phi^+}\langle K_{2\beta+\alpha},K_{2\beta+\alpha}^{-1}\rangle U^-[s_{\alpha}s_{\beta}]_{\phi^-}$$
with $\phi^+(E_{\alpha}K_{\alpha}^{-1})=\lambda_{\alpha}^+$ and $0$ otherwise, and $\phi^-(F_{\alpha})=\lambda_{\alpha}^-$ and $0$ otherwise, such that $\lambda_{\alpha}^+\lambda_{\alpha}^-=\frac{q^2}{(1-q^2)(q-q^{-1})}$ as above.

%Which combinations are possible in the case of '$ C $ is a Borel subalgebra' ', we will discuss below.

%\subsubsection{Konstruktion der Borelunteralgebren der $U_q(\sl_3)$}\label{kapitelborelsl3}
%Now we want to construct and classify the Borel subalgebras of $U_q(\sl_3)$. We start with the triangular Borel subalgebras. 
%Three types of triangular Borel subalgebras from Kapitel $6$, $7$ and $8$ are presented. Then we will prove, that there are no more types. 
%Therefore we consider different combinations of elements in EZS.
%For some combinations we will construct? a contradiction to the ede property with the minuscule Vermamodul.
%Then, we will prove with the maximality, that all triangular Borel subalgebras are of the described form. 
%Wir haben oben gesehen, dass im miniscule Fall 
%unterschiedliche Wurzelvektoren zur selben Wurzel bis auf Konstante gleich wirken, daher werden wir im Folgenden stellvertretend immer nur einen Typ Wurzelvektor betrachten, außerdem alle Kombinationen nur bis auf Symmetrie ausarbeiten.


\begin{thebibliography}{xxxxxx}
%\bibitem[AHS10]{AHS10} N. Andruskiewitsch, I. Heckenberger, H.-J. Schneider:
 %\emph{The Nichols algebra of a semisimple {Y}etter-{D}rinfeld module},
 %Amer. J. Math. 132 (2010) 1493-1547
 
 %\bibitem[BB05]{BB05}
%A.~Björner and F.~Brenti,\\
%\emph{Combinatorics of Coxeter Groups},
%Springer-Verlag, 2005.

\bibitem[HK11a]{HK11a}
I.~Heckenberger and S.~Kolb,\\
\emph{Homogeneous right coideal subalgebras of quantized enveloping
 algebras}, Bull. London Math. Soc. 44, 837-848, 2012.


\bibitem[HK11b]{HK11b}
I.~Heckenberger and S.~Kolb,\\
\emph{Right coideal subalgebras of the {B}orel part of a quantized
 enveloping algebra}, Int. Math. Res. Not. IMRN, no.~2, pp
 419-451, 2011.
 
 
% \bibitem[HLV17]{HLV17}
%I.~Heckenberger and S.~Lentner and K.~Vocke\\
%\emph{On Borel subalgebras of quantum groups},
%Int. Math. Res. Not. IMRN, no.~2, pp
% 419-451, 2011.
 
\bibitem[HS09]{HS09}
I.~Heckenberger and H.-J. Schneider,\\
\emph{Right coideal subalgebras of
 {N}ichols algebras and the {D}uflo order on the {W}eyl groupoid}, Israel Journal of Mathematics 197, 139-187, 2013.
 


%\bibitem[HS10]{HS10} I. Heckenberger, H.-J. Schneider:
 %\emph{Root systems and {W}eyl groupoids for {N}ichols algebras},
 %Proc. Lond. Math. Soc. 101 (2010) 623-654
 
 %\bibitem[Hum70]{Hum70}
%J.~Humphreys,\\
%\emph{Introduction to Lie Algebras and Representation Theory},
%Springer-Verlag, Library of Congress Catalog, 1980.

\bibitem[Jan96]{Jan96}
J.~Jantzen,\\
\emph{Lectures in Quantum Groups},
Graduate Studies in Mathematics, Volume 6, American Mathematical Society, 1996.

%\bibitem[Kolb14]{Kolb14} S. Kolb: \emph{Quantum symmetric Kac-Moody pairs}, Adv. Math. 267 (2014), 395-469. 

%\bibitem[Len14]{Len14}
%S.~Lentner,\\
%\emph{A Frobenius homomorphism for Lusztig's quantum groups for arbitrary roots of unity}
%Commun. Contemp. Math., Volume 18, 2016.

%\bibitem[Let97]{Let97} G. Letzter: \emph{Subalgebras which appear in quantum Iwasawa decompositions}, Canadian Journal of Mathematics 49 (1997), 1206--1223.

\bibitem[LS91]{LS91}
S.~Levendorskii and Y.~Soibelman,\\
\emph{Algebras of functions on compact quantum groups, Schubert cells and quantum tori}, 
Commun. Math. Phys. 139, no.1, 141-170, 1991.


%\bibitem[Lusz90]{Lusz90} G. Lusztig: \emph{Finite dimensional Hopf algebras
 %arising from quantized universal enveloping algebras}, J. Amer. Math. Soc. 3
 %(1990), 257-296. 

\bibitem[PA]{PA}
P. Papi,\\
\emph{A characterization of a special ordering in a root system}
Proceedings of the American Mathematical Society, Volume 120, pp 661-665, 1994.
 
 
\bibitem[Vocke16]{Vocke16}
K. Vocke,\\
\emph{Ueber Borelunteralgebren von Quantengruppen}
Dissertation University of Marburg.


%\bibitem[Vocke17]{Vocke17}
%K. Vocke,\\
%\emph{On right coideal subalgebras of quantum groups}
%...

\end{thebibliography}
\end{document}